\documentclass[11pt]{amsart}
\usepackage{dsfont}
\usepackage[usenames,dvipsnames,svgnames,table]{xcolor}
\usepackage[utf8]{inputenc}
\usepackage[T1]{fontenc}
\usepackage{ tipa }
\usepackage{helvet}
\usepackage{color}
\usepackage{mathrsfs}  
\usepackage{stmaryrd}  
\usepackage{amsmath,amsxtra,amsthm,amssymb,xr,enumerate,enumitem,fullpage,graphicx,mathtools}
\usepackage{fdsymbol}

\usepackage{footmisc}
\usepackage[margin=0.8in]{geometry}

\usepackage[all]{xy}
\usepackage[bbgreekl]{mathbbol}
\usepackage{bbm}
\usepackage{enumitem}

\DeclareMathSymbol{\invques}{\mathord}{operators}{`>}
\DeclareUnicodeCharacter{00BF}{\tmquestiondown}
\DeclareRobustCommand{\tmquestiondown}{%
  \ifmmode\invques\else\textquestiondown\fi
}
\usepackage{verbatim}
\numberwithin{equation}{section}

\makeatletter
\newcommand{\mylabel}[2]{#2\def\@currentlabel{#2}\label{#1}}
\makeatother

\newtheorem{theorem}{Theorem}[section]
\newtheorem{lemma}[theorem]{Lemma}
\newtheorem{conj}[theorem]{Conjecture}
\newtheorem{proposition}[theorem]{Proposition}
\newtheorem{corollary}[theorem]{Corollary}
\newtheorem{defn}[theorem]{Definition}

\newtheorem{remark}[theorem]{Remark}

\newcommand{\cI}{\mathcal{I}}

\newcommand{\fv}{\mathfrak{v}}

\newcommand{\cX}{\mathcal{X}}

\newcommand{\Gal}{\operatorname{Gal}}

\newcommand{\CC}{\mathbb{C}}

\newcommand{\QQ}{\mathbb{Q}}
\newcommand{\Qp}{\mathbb{Q}_p}
\newcommand{\Zp}{\mathbb{Z}_p}
\newcommand{\ZZ}{\mathbb{Z}}

\newcommand{\ord}{\mathrm{ord}}

\newcommand{\fp}{\mathfrak{p}}

\newcommand{\vp}{\varphi}

\newcommand{\cH}{\mathcal{H}}
\newcommand{\cO}{\mathcal{O}}

\newcommand{\HIw}{H^1_{\mathrm{Iw}}}

\newcommand{\GL}{\mathrm{GL}}

\newcommand{\col}{\mathrm{Col}}
\newcommand{\image}{\mathrm{Im}}
\newcommand{\cyc}{\textup{cyc}}

\newcommand{\Frob}{\mathrm{Frob}}

\newcommand{\Hom}{\mathrm{Hom}}
\newcommand{\Sel}{\mathrm{Sel}}
\newcommand{\Char}{\mathrm{Char}}

\newcommand{\cP}{\mathcal{P}}
\newcommand{\cF}{\mathcal{F}}

\newcommand{\fM}{\m}

\newcommand{\rd}{\mathrm{d}}

\definecolor{Green}{rgb}{0.0, 0.5, 0.0}

\newcommand{\green}[1]{\textcolor{Green}{#1}}

\newcommand{\m}{\mathfrak{m}}

\newcommand{\Div}{\mathrm{Div}}

\renewcommand{\P}{\wp}

\newcommand{\cG}{\mathcal{G}}

\newcommand{\cE}{\mathcal{E}}

\newcommand{\Cp}{\mathbb{C}_p}

\newcommand{\Up}{\Upsilon}

\newtheorem{lthm}{Theorem} 

\newtheorem{assumption}{Assumption}

\newenvironment{smallpmatrix}
  {\left(\begin{smallmatrix}}
  {\end{smallmatrix}\right)}

\usepackage{hyperref}

 \definecolor{pAlgae}{RGB}{87,115,135}
\definecolor{airforceblue}{rgb}{0.36, 0.54, 0.66}
	\definecolor{bondiblue}{rgb}{0.0, 0.58, 0.71}
\definecolor{britishracinggreen}{rgb}{0.0, 0.26, 0.15}
\definecolor{camouflagegreen}{rgb}{0.47, 0.53, 0.42}
\definecolor{darkcyan}{rgb}{0.0, 0.55, 0.55}

\hypersetup{
    colorlinks = true,
    linkcolor=blue,
     filecolor=blue,
     citecolor = darkcyan,      
     urlcolor=cyan,
    linkbordercolor = {white},
}

\begin{document}
\author{Raiza Corpuz}
\address{Raiza Corpuz\newline  Institute of Mathematics\\ University of the Philippines Diliman\\ Quezon City\\Philippines 1101}
\email{rcorpuz@math.upd.edu.ph}

\author{Antonio Lei}
\address{Antonio Lei\newline Department of Mathematics and Statistics\\University of Ottawa\\
150 Louis-Pasteur Pvt\\
Ottawa, ON\\
Canada K1N 6N5}
\email{antonio.lei@uottawa.ca}

\title{Congruences of $p$-adic $L$-functions of modular forms at non-ordinary primes}

\begin{abstract}
We present an analogue of Greenberg--Vatsal's and Emerton--Pollack--Weston's results on congruences of $p$-adic $L$-functions for $p$-non-ordinary cuspidal eigenforms $f$ and $g$ of equal weight that are $p$-congruent. In particular, we prove that the Iwasawa invariants of the analytic and algebraic signed $p$-adic $L$-functions of $f$ and $g$ are related by explicit formulae under appropriate hypotheses. We also show under the same assumptions that provided the algebraic and analytic $\mu$-invariants vanish, the signed Iwasawa main conjecture is true for $f$ if and only if it is true for $g$.
\end{abstract}

\maketitle

\section{Introduction} 

Let $p$ be an odd prime. Let $f$ and $g$ be Hecke eigenforms such that their residual representations at $p$ are isomorphic: $\overline{\rho}_f  \cong \overline{\rho}_g$. Suppose that $f$ and $g$ are $p$-ordinary. In the seminal work of Vatsal \cite{Vat99}, he observed a congruence between the $p$-adic $L$-functions of such modular forms. The main conjecture of Iwasawa theory postulates an equality between the ideal generated by the $p$-adic $L$-function and the characteristic ideal of the dual $p$-primary Selmer group over the cyclotomic $\Zp$-extension of $\QQ$. Let $\mu(\ast)_\text{\rm an}$ and $\lambda(\ast)_\text{\rm an}$ (resp. $\mu(\ast)_\text{\rm alg}$ and $\lambda(\ast)_\text{\rm alg}$) be the analytic (resp. algebraic) Iwasawa invariants for $\ast\in\{f,g\}$. These invariants measure the $p$-divisibility of a power series and the number of zeroes in the $p$-adic open unit disk. The Iwasawa invariants of $f$ and $g$ are expected to be closely linked.

When $f$ and $g$ correspond to elliptic curves, Greenberg and Vatsal \cite{greenbergvatsal} showed that the Iwasawa invariants of $f$ and $g$ behave in a parallel manner. This phenomenon was also observed by Emerton, Pollack and Weston \cite{EPW} within a Hida family $\mathcal{H}(\overline{\rho})$ that parametrizes modular forms with the same residual representation $\overline{\rho}$. For $\bullet \in \{ \text{\rm an, alg} \}$, these authors have proven that if $\mu(f_0)_\bullet = 0$ for one $f_0 \in \mathcal{H}(\overline{\rho})$, then $\mu(f)_\bullet = 0$ for any $f \in \mathcal{H}(\overline{\rho})$. They have also shown that if the $\mu_\bullet$-invariant for one element of $\cH(\overline\rho)$ (and therefore for all) vanishes, then one obtains a transition formula for the $\lambda_\bullet$-invariant. Lastly, combined with Kato's result on one divisibility of the Iwasawa main conjecture in \cite{kato04}, they proved that if the Iwasawa main conjecture is true for some $f_0\in\cH(\overline\rho)$ and the $\mu_{\mathrm{alg}}$- and $\mu_{\mathrm{an}}$-invariants vanish, then the same is true for all $f\in\cH(\overline\rho)$.

It is natural to ask whether the results discussed above can be extended to a pair of \emph{$p$-non-ordinary} Hecke eigenforms. There are points of divergence from the ordinary case: first, the $p$-adic $L$-functions  constructed in \cite{amicevelu75,visik76,MTT} do not give rise to bounded power series; second, the dual $p$-primary Selmer groups over the cyclotomic $\Zp$-extension of $\QQ$ are not torsion over the corresponding Iwasawa algebra.

The last few decades have seen substantial progress toward bringing the non-ordinary case in equal footing as the ordinary case. For a modular form $f$ of weight $k \geq 2$ and  $a_p(f) = 0$, Pollack \cite{pollack03} constructed plus and minus $p$-adic $L$-functions $L_p^\pm(f)$ that give rise to bounded power series. Let $\Up$ be a root of the polynomial $X^2 + a_p(f)X + \epsilon_f(p)p^{k-1}$ and write $L_p(f, \Up)$ for the $p$-adic $L$-function given in  \cite{amicevelu75,visik76,MTT}. We have
\[
L_p(f, \Up) = \log_{p,k}^+ \cdot L_p^+(f) + \Up \cdot \log_{p,k}^- \cdot L_p^-(f),
\]
where $\log_{p,k}^\pm$ are explicit functions defined using cyclotomic polynomials. Generalizing this, Sprung \cite{sprung09,Sprung17} obtained a pair of $p$-adic functions $(L_p(f, \sharp), L_p(f, \flat))$ for modular forms of weight two without the assumption that $a_p(f)=0$. When $f$ corresponds to an elliptic curve with $a_p(f)=0$, Kobayashi gave an arithmetic interpretation of Pollack's signed $p$-adic $L$-functions \cite{kobayashi03} by constructing the plus and minus Selmer groups $\Sel_p^\pm(f)$. The duals of these groups are torsion over the Iwasawa algebra of the cyclotomic $\Zp$-extension of $\QQ$, allowing us to formulate the plus and minus Iwasawa main conjectures, which assert that there exists $n^\pm \in \ZZ$ such that the characteristic ideal of $\Sel_p^\pm(f)^\vee$ is generated by $p^{n^\pm} \cdot L_p^\pm(f)$.
One inclusion of this conjecture was proved in \cite{kobayashi03}. When the elliptic curve has complex multiplication, the full main conjecture was proved by Pollack--Rubin \cite{pollackrubin04}.
These ideas have been generalized to higher-weight modular forms in \cite{lei09,LLZ0,LLZ0.5} using $p$-adic Hodge theory and Wach modules. 

The main goal of this article is to produce an analogue of Greenberg--Vatsal's and Emerton--Pollack--Weston's results for $p$-non-ordinary cuspidal eigenforms.
We fix once and for all a prime $p \geq 3$ and an embedding $\overline{\QQ} \hookrightarrow \Cp$. We consider a pair of normalized cuspidal eigenforms $f = \sum a_n q^n \in S_k(\Gamma_1(N_f), \epsilon_f)$ and $g = \sum b_n q^n \in S_k(\Gamma_1(N_g), \epsilon_g)$ that are non-ordinary at $p$ with $\mathrm{lcm}(N_f,N_g)\ge4$. By non-ordinary, we mean that $\ord_p(a_p), \ord_p(b_p) > 0$. Let $E$ be a finite extension of $\Qp$ that contains the image of $a_n$ and $b_n$ for all $n$ under our fixed embedding, and let $\mathcal{O}$ be its ring of integers. We choose a uniformizer $\P$ of the maximal ideal of $\mathcal{O}$.

Let $\omega: \Gal(\QQ(\mu_p)/\QQ) \to \Zp^\times$ be the Teichm\"{u}ller character.
We write $L_p(\ast, \natural, \omega^i, X)\in\cO\llbracket X\rrbracket\otimes_\cO E$ for the $\omega^i$-isotypic component of the signed $p$-adic $L$-function $L_p(\ast,\natural)$ for $\ast\in\{f,g\}$ and $\natural\in\{\sharp,\flat\}$ (see \S\ref{sec:2} for a detailed discussion).
 In this article, we work with the following assumptions.

\begin{assumption}\label{asi:main}
\textit{(analytic congruence)}
    The eigenforms $f$ and $g$ are $\P^r$-congruent, that is, there exists a positive integer $r$, such that $a_m \equiv b_m \bmod \P^r$, for all integers $m$ that are relatively prime to $N$.
\end{assumption}


\begin{assumption}\label{asi:FL}
\textit{(Fontaine--Laffaille)} 
$p > k$.
\end{assumption}

\begin{assumption}\label{asi:unram}
\textit{(unramified)}
The extension $E/\Qp$ is unramified.
\end{assumption}

\begin{assumption}\label{asi:Lp-nonzero}
\textit{(non-zero)} $\natural\in\{\sharp,\flat\}$ and $i\in\{0,\dots, p-2\}$ are such that both $L_p(f, \natural, \omega^i, X)$ and $L_p(g, \natural, \omega^i, X)$ are non-zero.\end{assumption}

\begin{assumption}\label{asi:cong}
\textit{(algebraic congruence)}
$T_f/\P T_f \cong T_g/\P T_g$ as $G_\QQ$-representations.
\end{assumption}

In Assumption~\ref{asi:cong}, $T_f$ and $T_g$ are $G_\QQ$-stable $\cO$-lattices inside Deligne's representations attached to $f$ and $g$, respectively. Since $f$ and $g$ are assumed to be $p$-non-ordinary, these representations are irreducible and $G_\QQ$-stable lattices are unique up to isomorphism.
Note that we may take $\P=p$ under Assumption~\ref{asi:unram}. Furthermore, if Assumption~\ref{asi:cong} holds, then $a_m\equiv b_m \bmod p$ for all $m$ that are relatively prime to $N$. In other words, Assumption~\ref{asi:cong} implies that Assumption~\ref{asi:main} holds for $r=1$. Conversely, under Assumptions \ref{asi:main} and \ref{asi:FL}, the residual representations are irreducible (see \cite[Theorem~2.6]{edixhoven92}); this implies that Assumption~\ref{asi:cong} holds. For our purposes, we prove a congruence relation between signed $p$-adic $L$-functions under Assumptions~\ref{asi:main}-\ref{asi:unram}. We then study applications to Iwasawa theory under Assumptions~\ref{asi:FL}-\ref{asi:cong}.

Our strategy is as follows. In Section~\ref{sec:1}, we recall the relationship between modular symbols and special values of the complex $L$-function associated with modular forms. We make use of Assumptions~\ref{asi:main} and \ref{asi:FL} to show that the congruence between $f$ and $g$ translates to a congruence between their respective modular symbols after choosing appropriate periods. Towards the end of the section, we recall the definition of Mazur--Tate elements. In Section~\ref{sec:2}, we review how the Mazur--Tate elements for different critical twists are used to construct the $p$-adic $L$-functions in \cite{amicevelu75,visik76,MTT}. We then explain how to decompose them into signed $p$-adic $L$-functions $L_p(f, \sharp, \omega^i, X)$ and $L_p(f, \flat, \omega^i, X)$ using the logarithm matrix introduced in \cite{BFSuper}. Furthermore, we give a detailed analysis of the integrality of these elements under Assumptions~\ref{asi:FL} and \ref{asi:unram}. 
In Section~\ref{sec:3}, we demonstrate how the congruence we obtained in Section~\ref{sec:1} trickles down to a pair of congruences:
\begin{align*}
    L_{p,\Sigma_0}(f,\sharp,\omega^i,X) &\equiv L_{p,\Sigma_0}(g,\sharp,\omega^i,X) \,\bmod\,\P^{r} \cdot \mathcal{O}\llbracket X\rrbracket,\\
    L_{p,\Sigma_0}(f,\flat,\omega^i,X) &\equiv L_{p,\Sigma_0}(g,\flat,\omega^i,X) \,\bmod\,\P^{r} \cdot \mathcal{O}\llbracket X\rrbracket,
\end{align*}
where $\Sigma_0$ is the set of primes dividing $N_fN_g$ and $L_{p,\Sigma_0}$ represents the $\Sigma_0$-imprimitive $p$-adic $L$-function. The reader is referred to Theorem~\ref{thm:cong-Lp} for a precise statement. This result allows us to deduce the following theorem on Iwasawa invariants (here, the Iwasawa invariants of $L_p(\ast,\natural,\omega^i,X)$ are denoted by $\mu(\ast,\omega^i)^\natural_\mathrm{an}$ and $\lambda(\ast,\omega^i)^\natural_\mathrm{an}$) in Section~\ref{sec:4}. 

\begin{lthm}(Theorem~\ref{thm:analytic})\label{thmA}
Let $f$ and $g$ be $p$-non-ordinary modular forms, $i\in\{0,\dots, p-2\}$ and $\natural \in \{ \sharp, \flat \}$ satisfying Assumptions~\ref{asi:FL}--\ref{asi:cong}.  Let $\Sigma_0$ be the finite set of primes dividing $N_f N_g$. Then $\mu(f, \omega^i)_\text{\rm an}^\natural=0$ if and only if $\mu(g, \omega^i)_\text{\rm an}^\natural=0$, in which case
\[ 
\lambda(f, \omega^i)_\text{\rm an}^\natural = \lambda(g, \omega^i)_\text{\rm an}^\natural + \sum_{v \in \Sigma_0} \left(\mathbf{e}_v (g, \omega^i) - \mathbf{e}_v (f, \omega^i)\right), 
\]
where each $\mathbf{e}_v(\ast,\omega^i)$ corresponds to the $\lambda$-invariant of an Iwasawa function interpolating the Euler factor of the $L$-function of $\ast$ at $v \in \Sigma_0$.
\end{lthm}

\noindent The latter half of Section~\ref{sec:4} is spent recalling the signed Selmer groups of $p$-non-ordinary modular forms and their respective duals, denoted by $\mathcal{X}^\natural(\ast)$, for $\natural\in\{\sharp,\flat\}$ and $\ast\in\{f,g\}$. The Iwasawa invariants of its $\omega^i$-isotypic component are denoted by $\mu(\ast,\omega^i)^\natural_\mathrm{alg}$ and $\lambda(\ast,\omega^i)^\natural_\mathrm{alg}$. We extrapolate an algebraic counterpart of Theorem~\ref{thmA}, which follows from the results in \cite{HL19}.

\begin{lthm} (Theorem~\ref{thm:algebraic})\label{thmB}
Let $f$, $g$, $\natural$, $i$ and $\Sigma_0$ be as in Theorem~\ref{thmA}.  Then, $\mathcal{X}^\natural(f)^{\omega^i}$ and $\mathcal{X}^\natural(f)^{\omega^i}$ is $\cO\llbracket X\rrbracket$-torsion. In addition, $\mu(f, \omega^i)^\natural_\text{\rm alg}=0$ if and only if $\mu(g, \omega^i)^\natural_\text{\rm alg}=0$, in which case
\[ 
\lambda(f,\omega^i)_\text{\rm alg}^\natural = \lambda(g,\omega^i)_\text{\rm alg}^\natural + \sum_{\mathfrak{v} \in \Sigma}\left( \mathbf{e}_{v}(g,\omega^i)- \mathbf{e}_{v}(f,\omega^i)\right).
\]
\end{lthm}
Together, these two theorems can be utilized to give the following application to the signed Iwasawa main conjectures (see Conjecture~\ref{IMC}).

\begin{lthm}[Theorem~\ref{thm:IMC}]\label{thmC}
    Let $f$, $g$, $\natural$, $i$ and $\Sigma_0$ be as in Theorem~\ref{thmA}. If $\mu(\ast, \omega^i)^\natural_\text{\rm alg}=\mu(\ast, \omega^i)^\natural_\text{\rm an}=0$ for $\ast\in\{f,g\}$, then the signed Iwasawa main conjecture holds for $\cX^\natural(f)^{\omega^i}$ if and only if it holds for $\cX^\natural(g)^{\omega^i}$.
\end{lthm}

Similar questions were studied in \cite{KLP}, where the authors showed that modules involving Kato's zeta elements that are used to formulate the Iwasawa main conjecture without $p$-adic $L$-functions behave well under congruences. In particular, they proved that the Iwasawa invariants of these modules satisfy similar relations to those given by Theorems~\ref{thmA} and \ref{thmB} when $f$ and $g$ have isomorphic residual representations. Furthermore, they obtained applications to the signed Iwasawa main conjectures, in a similar manner as Theorem~\ref{thmC}. Our method is different, and we work under different hypotheses. Finally, we remark that in the special case of elliptic curves, Theorem~\ref{thmB} was already established by Kim in \cite{kim09AJM}. Other similar results were studied in \cite{kim21,choikim,raysujatha,hamidi,AhmedLim,ponsinet20,HLV,kimsujatha} in various contexts.


\subsection{Notation}

Let $\mu_{p^n}$ be the multiplicative group of $p^n$-th roots of unity. For $n \geq 1$, we set $\mathcal{G}_n = \Gal(\QQ(\mu_{p^n})/\QQ)$. The fields $\QQ(\mu_{p^n})$ form a tower giving $\QQ_\infty = \QQ(\mu_{p^\infty}) = \bigcup_n \QQ(\mu_{p^n})$ with Galois group 
\[
\mathcal{G}_\infty \cong \Zp^\times \cong \Delta \times G_\infty,
\]
where $\Delta=\Gal(\QQ(\mu_p)/\QQ) \cong (\ZZ/p\ZZ)^\times$ and $G_\infty=\Gal(\QQ(\mu_{p^\infty})/\QQ(\mu_p))\cong \Zp$. We denote by $\QQ_\text{\rm cyc}$ the cyclotomic $\Zp$-extension of $\QQ$ with Galois group isomorphic to $G_\infty$. 
We define the Iwasawa algebra of $G_\infty$ as $\Lambda = \mathcal{O}\llbracket G_\infty \rrbracket = \varprojlim \cO[G_n]$, where $G_n=\Gal(\QQ(\mu_{p^{n+1}})/\QQ(\mu_p))\cong\ZZ/p^n\ZZ$. Let $\gamma$ be a topological generator of $G_\infty$. One can view the Iwasawa algebra $\cO\llbracket G_\infty\rrbracket$ as a power series ring $\mathcal{O}\llbracket X \rrbracket$ by sending $\gamma$ to $1 + X$. This induces an isomorphism $\cO[G_n] \cong \cO\llbracket X \rrbracket / \left( (1 + X)^{p^n} - 1 \right)$. We define the Iwasawa algebra $\cO\llbracket \mathcal{G}_\infty \rrbracket$ similarly. This in turn can be identified with the power series ring $\cO[\Delta]\llbracket X \rrbracket$.

Let $\omega: \Gal(\QQ(\mu_p)/\QQ) \to \Zp^\times$ be the Teichm\"{u}ller character. The group $\Hom(\Delta, \Zp^\times)$ consists of $\omega^i$ for $0 \leq i \leq p-2$. Let $\theta \in \Hom(\Delta, \Zp^\times)$. We write $e_\theta = \tfrac{1}{p-1} \sum_{\tau \in \Delta} \theta(\tau) \cdot \tau^{-1}$ for the corresponding idempotent element of $\cO[\Delta]$. Given an $\cO\llbracket \mathcal{G}_\infty \rrbracket$-module $M$, we denote by $M^\theta = e_\theta \cdot M$ the $\theta$-isotypic component of $M$. 

\subsection*{Acknowledgement}
The authors thank Daniel Delbourgo, Rylan Gajek-Leonard, Jeffrey Hatley and Chan-Ho Kim for interesting discussions during the preparation of this article. The authors also thank the anonymous referee for helpful comments and suggestions on an earlier version of the article. AL's research is supported by the NSERC Discovery Grants Program RGPIN-2026-04351. Parts of this work were carried out during RC's visit to the University of Ottawa in 2024 through the VRS scheme. On behalf of all authors, the corresponding author states that there is no conflict of interest. This manuscript has no associated data.


\section{Review of modular symbols and Mazur--Tate elements}\label{sec:1}

\subsection{Modular symbols attached to cusp forms}\label{1.1}
We review the definition of modular symbols attached to $f$; the same discussion applies to $g$.

 Let $\textbf{D} = \Div(\mathbb{P}^1(\QQ))$ be the free abelian group of divisors supported on the rational cusps $\mathbb{P}^1(\QQ)$ and let $\textbf{D}^0 \subseteq \textbf{D}$ be its subgroup of degree 0 divisors. Then $\textbf{D}^0$ is generated by the symbols $[r] - [s]$ where $r,s\in\mathbb{P}^1(\QQ)$. Let $R$ be a commutative ring with unity and let $\Gamma$ be a congruence subgroup. Let $A$ be a right $R[\Gamma]$-module. We call an additive homomorphism $\xi: \textbf{D}^0 \to A$ that satisfies $\xi(\gamma D) | \gamma = \xi(D)$ for all $\gamma \in \Gamma$ an \emph{$A$-valued modular symbol}. We denote by $\text{\rm Symb}(\Gamma, A)$ the $R$-module of modular symbols.

The group of matrices $\GL_2(R)$ acts on the space $L_{k-2}(R)$ of degree $k-2$ homogeneous polynomials defined over $R$ in the indeterminates $X$ and $Y$ as follows: for $\gamma = \begin{smallpmatrix}
    a & b\\ c & d
\end{smallpmatrix} \in \GL_2(R)$,
\begin{equation}\label{actiononpoly}
    P(X,Y) | \gamma := P\left( (X, Y) \gamma^\iota \right) = P(dX - cY, -bX + aY).
\end{equation}
Consider a cusp form $f \in S_k(\Gamma)$. One may view $[r]-[s]$ as any simple path $s \to r$ in the upper half plane which allows us to construct a $\CC$-valued modular symbol
\begin{equation}\label{eq:modsymbc}
  \xi_f: [r] - [s] \mapsto 2\pi \sqrt{-1}\int_s^r f(z)(zX+Y)^{k-2}\,\rd z
\end{equation}
in $\text{\rm Symb}(\Gamma, L_{k-2}(\CC)) := \Hom_{\Gamma}(\textbf{D}^0, L_{k-2}(\CC))$. 

The modular symbol $\xi_f$ encodes the special values of the $L$-function $L(f,s)$ associated with $f$. For a positive integer $D$, let $\chi: (\ZZ/D\ZZ)^\times \to \CC^\times$ be a Dirichlet character. It is well-known that the $L$-function of the twisted form $f_{\overline{\chi}} = \sum \overline{\chi}(n) a_nq^n$ can be written as
\[ 
L(f, \overline{\chi}, s) = \frac{(-2\pi \sqrt{-1})^s}{\Gamma(s) \tau(\chi)} \sum_{a \in (\ZZ/D\ZZ)^\times} \chi(a) \int_{a/D}^\infty f(z)\left( z - \frac{a}{D} \right)^{s-1}\rd z,
 \]
where $\Gamma(s)$ denotes the Gamma function and $\tau(\chi)$ is the Gauss sum given by $\sum_{a\in(\ZZ/D\ZZ)^\times} \chi(a) e^{2\pi \sqrt{-1} a/D}$. The action in \eqref{actiononpoly} results in
\[
\xi_f \bigg| \begin{pmatrix}
    1 & a\\ 0 & D
\end{pmatrix} (\{\infty\} - \{0\}) 
= 2 \pi \sqrt{-1} \int_{a/D}^\infty f(z)(D zX + (-aX + Y))^{k-2}\rd z. 
\]
Setting $(X,Y) = (1,1)$, the right-hand side becomes
\[ 
2 \pi \sqrt{-1} \sum_{j=0}^{k-2} \binom{k-2}{j}\,D^j \int_{a/D}^\infty f(z)\left(z - \frac{a}{D} \right)^j\rd z.
 \]
For $0 \leq j \leq k-2$, let $\xi_{f,j}(a,D)$ be the projection of the above expression to the $j$-th exponent. We have:
\begin{equation}\label{eq:specialvalue}
L(f, \overline{\chi}, j+1) = \frac{(-2\pi \sqrt{-1})^{j}}{j! \tau(\chi)} \cdot \frac{1}{D^j} \cdot \binom{k-2}{j}^{-1} \sum_{a \in (\ZZ/D\ZZ)^\times} \chi(a)\,\xi_{f,j}(a,D).
\end{equation}


\subsection{The Eichler-Shimura isomorphism}\label{sec:1.2}
Let $\Gamma = \Gamma_1(N)$, where $N \geq 4$ is a common multiple of $N_f$ and $N_g$ so that $\Gamma$ is a discrete torsion-free subgroup of $\text{\rm SL}_2(\mathbb{Z})$. 

For a $\Gamma$-module $A$, let $H^q(\Gamma, A)$ be the $q$-th cohomology group and denote by $H_c^q(\Gamma, A)$ those with compact support. We define the parabolic cohomology group $H_P^1(\Gamma, A)$ as the natural image of $H_c^q(\Gamma, A) \to H^q(\Gamma, A)$. Let $R$ be a ring as defined above and set $A = L_{k-2}(R)$.  It follows from \cite[Proposition~4.2]{ashstevens} that there is a canonical isomorphism
\[ 
H_c^1(\Gamma, L_{k-2}(R)) \cong \text{\rm Symb}(\Gamma, L_{k-2}(R)),
 \]
which results in an exact sequence
\begin{equation}\label{eq:exseq1}
S_k(\Gamma) \stackrel{f \mapsto \xi_f}{\longrightarrow} \text{\rm Symb}(\Gamma, L_{k-2}(\CC)) \stackrel{\sim}{\longrightarrow} H_c^1(\Gamma, L_{k-2}(\CC)) \longrightarrow H^1(\Gamma, L_{k-2}(\CC)).
\end{equation}
The following theorem is commonly referred to as the Eichler-Shimura isomorphism. It relates the $\CC$-vector space of modular forms to cohomology groups, allowing us to analyze the former using techniques in cohomology theory.

\begin{theorem}\label{thm:eichlershimura}
There is an isomorphism 
\[ H_P^1(\Gamma, L_{k-2}(\CC)) \cong S_k(\Gamma) \oplus S_k(\Gamma)^c \]
which extends to
\begin{equation}\label{eq:eichshi2}
H^1(\Gamma, L_{k-2}(\CC)) \cong M_k(\Gamma) \oplus S_k(\Gamma)^c
\end{equation}
where $c $ denotes the complex conjugation operator. 
\end{theorem}
\begin{proof}
Note that $S_k(\Gamma) \otimes_\mathbb{R} \CC \cong S_k(\Gamma) \oplus S_k(\Gamma)^c$. One can then read off the first isomorphism from the exact sequence \eqref{eq:exseq1}. As for the second isomorphism, we recall the decomposition of $M_k(\Gamma) \cong S_k(\Gamma) \oplus E_k(\Gamma)$ into the space $S_k(\Gamma)$ of cusp forms and the space $E_k(\Gamma)$ of Eisenstein series. The claim follows using a dimension formula argument; see \cite[\S6.2]{hidabook} for details.
\end{proof}

We show that given $f$ and $g$ satisfying Assumptions~\ref{asi:main} and \ref{asi:FL} as in the introduction, one can observe a congruence between their associated modular symbols. To do so, we first bring the discussion to an integral setting.

\begin{lemma} \label{lem:tensor}
If $A$ is an $R$-flat algebra, then $H_P^1(\Gamma, L_{k-2}(\CC)) = H_P^1(\Gamma, L_{k-2}(A)) \otimes_R \CC$. 
\end{lemma}

\noindent Lemma~\ref{lem:tensor} is proved in \cite[\S6.2 (1b) on p.168]{hidabook}. Let $E$ be as in the introduction. Lemma~\ref{lem:tensor}, combined with Theorem~\ref{thm:eichlershimura}, yields
\begin{align}\label{eq:eichshiO}
\begin{split}
H_P^1(\Gamma, L_{k-2}(E)) \otimes \CC &= H_P^1(\Gamma, L_{k-2}(\CC))\\
	&\cong S_k(\Gamma) \oplus S_k(\Gamma)^c\\
	&\cong (S_k(\Gamma, E) \otimes \CC) \oplus (S_k(\Gamma, E)^c \otimes \CC)\\
	&= \left (S_k(\Gamma, E) \oplus S_k(\Gamma, E)^c \right) \otimes \CC,
\end{split}
\end{align}
where the last equality follows from \cite[\S6.3, Theorem 2]{hidabook}. 

\begin{remark}\label{CtoO}
\textup{Let $f \in S_k(\Gamma, \mathcal{O})$. The $L_{k-2}(\CC)$-valued modular symbol $\xi_f$ can be decomposed as $\xi_f^++\xi_f^-$, where  $\xi_f^\pm$ belongs to the $\pm1$-eigenspace of the complex conjugation. 
Lemma~\ref{lem:tensor} implies that we can divide $\xi_f^\pm$ by a complex number $\Omega_f^\pm$ to obtain an $L_{k-2}(E)$-valued modular symbol $\varphi_f^\pm$. In fact, we can choose $\Omega_f^\pm$ so that $\varphi_f^\pm \in H_P^1(\Gamma, L_{k-2}(\cO))^\pm$. We refer to these numbers as \emph{cohomological periods}, which are well-defined up to units in $\cO$ (see \cite[Definition~2.1 and Remark~2.2]{pollack-weston11}). }
\end{remark}

Let $\textbf{T}_k(\Gamma, \mathcal{O})$ be the $\cO$-algebra generated by all Hecke operators $T_\ell$ acting faithfully on the space of modular forms $M_k(\Gamma, \mathcal{O})$ and let $\textbf{t}_k(\Gamma, \mathcal{O})$ be its quotient acting on cusp forms $S_k(\Gamma, \mathcal{O})$. It is well known that these algebras act on the torsion-free part of $H^1(\Gamma, L_{k-2}(\mathcal{O}))$ and $H_P^1(\Gamma, L_{k-2}(\mathcal{O}))$, respectively. The bilinear pairing
\[
S_k(\Gamma, \mathcal{O}) \times \textbf{t}_k(\Gamma, \mathcal{O}) \to \mathcal{O}
\]
defined by $(T_n, f) = a_1(f|T_n)$ induces a duality $\Hom_\mathcal{O}(\textbf{t}_k(\Gamma, \mathcal{O}),\cO) \cong S_k(\Gamma, \mathcal{O})$. The eigenforms $f$ and $g$ give rise to a group homomorphism
\[
\lambda_f, \lambda_g: \textbf{t}_k(\Gamma, \mathcal{O}) \to \mathcal{O}.
\]
 As $f$ and $g$ satisfy Assumption~\ref{asi:main}, the compositions of $\lambda_f$ and $\lambda_g$ with the reduction map modulo $\P$ coincide. We denote their common kernel by $\mathfrak{m}$. This is a maximal ideal in $\textbf{t}_k(\Gamma, \mathcal{O})$ and it determines an identical semisimple representation for $f$ and $g$
$$\rho_\mathfrak{m}: \Gal(\overline{\QQ}/\QQ) \to \GL_2(\textbf{t}_k(\Gamma, \mathcal{O})/\mathfrak{m})$$
that is unramified outside $Np$, such that for all primes $\ell \nmid Np$,
\begin{gather}
\begin{cases}
    \text{\rm Tr}(\rho(\Frob_\ell)) &\hspace{-0.3cm}= T_\ell \bmod \mathfrak{m}\\
    \det(\rho(\Frob_\ell)) &\hspace{-0.3cm}= \ell\langle\ell\rangle \bmod \mathfrak{m}.
\end{cases}
\end{gather}

Let $f=\sum_{n\ge1} a_n q^n$ and $g=\sum_{n\ge1} b_n q^n$ be normalized eigenforms as in the introduction. Consider any finite set of primes $\Sigma$ and let
\[
F := f_\Sigma = \sum_{n\ge1} a'_n q^n \qquad\text{and}\qquad G := g_\Sigma = \sum_{n\ge1} b_n' q^n,
\]
where $a_n' = b_n' = 0$ for integers $n$ divisible by primes in $\Sigma$, and $a_n'=a_n$, $b_n'=b_n$ otherwise. We call these the \emph{$\Sigma$-depletion} of $f$ and $g$ respectively.

 Let $v$ be a prime and define
\begin{equation}
\cP_v(f, X) = \det\left(1 - X \cdot \Frob^{-1}_v \Big| V_\ell(f)^{\mathcal{I}_{\QQ_v}} \right), \qquad (\ell \neq v).\label{eq:Euler}    
\end{equation}
Here, $V_\ell(f)$ is the $\ell$-adic $G_\QQ$-representation of Deligne, $\Frob_v$ denotes the arithmetic Frobenius at $v$ and $\mathcal{I}_{\QQ_v}$ is the inertia group. The complex $L$-function of $f$ is given by
\[
L(f,s) = \prod_{v \mid N_f} (1 - a_v(f) v^{-s})^{-1} \times \prod_{v \nmid N_f} \left( 1 - a_v(f) v^{-s} + \varepsilon_f(v) v^{k-1-2s} \right)^{-1}=\prod_v \cP_v(f,v^{-s})^{-1}
\]
for $\mathrm{Re}(s)\gg0$.
We define the \emph{$\Sigma$-imprimitive} $L$-function of $f$ as the usual $L$-function with the Euler factors removed at all primes $v \in \Sigma$
\[
L_{\Sigma}(f,s) = L(f,s) \times \prod_{v \in \Sigma} \cP_v(f, v^{-s}).
\]
We define $L_{\Sigma}(g,s)$ analogously.

Let $\Sigma_0$ be the set of primes dividing $N_f N_g$. The Hecke eigenvalues of the $\Sigma_0$-depletions $F = f_{\Sigma_0}$ and $G = g_{\Sigma_0}$ agree on every positive integer $n$. Let $M$ be an integer so that $F$ and $G$ are both modular forms of level $M$. Note that $M$ is divisible only by primes in $\Sigma_0$ because depleting a modular form, say at a prime $q$, increases the level by at most $q^2$. We henceforth set $\Gamma = \Gamma_1(M)$. 

The main result of this section is that if Assumptions~\ref{asi:main} and \ref{asi:FL} hold, then the special $L$-values of $F$ and $G$ are congruent modulo $\P^r$. Let $\xi_F$ and $\xi_G$ be the modular symbols defined by \eqref{eq:modsymbc}. As in Remark~\ref{CtoO}, we normalize them by their cohomological periods to define the $L_{k-2}(\mathcal{O})$-valued modular symbols $\varphi_F^\pm$ and $\varphi_G^\pm$ satisfying
\[ 
\varphi^\pm_\ast = \frac{\xi^\pm_\ast}{\Omega^\pm_\ast}, \quad\ast \in \{ F, G \}.
\]
We show that $\varphi^+_F$ and $\varphi^-_F$ are congruent to $\varphi^+_G$ and $\varphi^-_G$ modulo $\P^r$, respectively. To prove this, we will be needing the following multiplicity one result which is a consequence of \cite[Theorem~2.1]{FJ95} provided $p \nmid M$ and $p > k$.
\begin{equation}\label{eq:mult-one}
	H_P^1(\Gamma, L_{k-2}(\cO))^\pm_\mathfrak{m} \cong S_k(\Gamma, \cO)_\mathfrak{m}
\end{equation} 

The congruence class of $F$ and $G$ determines a maximal ideal $\mathfrak{m} \subset \textbf{t}_k(\Gamma, \mathcal{O})$ so that $F$ and $G$ lie in the same class in $S_k(\Gamma, \mathcal{O})_\mathfrak{m}$. Applying \cite[(6.7)]{greenbergstevens} with $A = \mathcal{O}$ and localizing at $\mathfrak{m}$, we deduce the following short exact sequence
\[ 0 \to \text{\rm B}(\Gamma, L_{k-2}(\mathcal{O}))_\mathfrak{m} \to \text{\rm Symb}(\Gamma, L_{k-2}(\mathcal{O}))_\mathfrak{m} \stackrel{\delta}{\to} H_P^1(\Gamma, L_{k-2}(\mathcal{O}))_\mathfrak{m} \to 0. \] 
The modular symbols $\varphi^\pm_\ast$ are sent to  $\delta(\varphi^\pm_\ast) \in H_P^1(\Gamma, L_{k-2}(\mathcal{O}))^\pm_\mathfrak{m}$ and since $F \equiv G\,\bmod\,\P^r S_k(\Gamma, \cO)_\mathfrak{m}$, we have $\delta(\varphi_F^\pm) \equiv \delta(\varphi_G^\pm)\,\bmod\,\P^r H_P^1(\Gamma, L_{k-2}(\cO))_\mathfrak{m}$. This congruence lifts naturally to $\varphi_\ast^\pm$ in $\text{\rm Symb}(\Gamma, L_{k-2}(\mathcal{O}))_\mathfrak{m}$. This is because the action of Hecke operators on the boundary symbols $\text{\rm B}(\Gamma, L_{k-2}(\mathcal{O}))_\mathfrak{m}$ are well-known to be Eisenstein, and the irreducibility of $\rho_\mathfrak{m}$ guarantees that the maximal ideal $\mathfrak{m}$ is non-Eisenstein. 

Therefore, we have the following congruence between modular symbols:
\begin{equation}\label{eq:congvarphi}
\sum_{a \in (\ZZ/D\ZZ)^\times} \varphi_F \bigg| \begin{pmatrix} 1 & -a\\ 0 & p^n \end{pmatrix} ( \{\infty\} - \{0\} ) \equiv \sum_{a \in (\ZZ/D\ZZ)^\times} \varphi_G \bigg| \begin{pmatrix} 1 & -a\\ 0 & p^n \end{pmatrix} ( \{\infty\} - \{0\} )\,\bmod\,\P^{r}.
\end{equation}
Since $p \nmid M$ and $p > k$, we can use Ihara's Lemma to relate the periods of the modular forms $f$ and $g$ to the periods of their respective $\Sigma_0$-depletions $F$ and $G$ by $\Omega_f^\pm = u_f^\pm \cdot \Omega_F^\pm$ and $\Omega_g^\pm = u_g^\pm \cdot \Omega_G^\pm$, where $u_f^\pm$, $u_g^\pm$ are $p$-adic units. We refer the reader to \cite[Proposition~1.4(c)]{DFG04} for the full statement of Ihara's Lemma. Without loss of generality, we can take $u_f^\pm=u_g^\pm=1$.

\begin{proposition}\label{prop:L-cong}
Let $\chi: (\ZZ/D\ZZ)^\times \to \CC^\times$ be a Dirichlet character and suppose $f$ and $g$ satisfy Assumptions~\ref{asi:main} and \ref{asi:FL}. Let $\Sigma_0 $ be the set of primes dividing $N_fN_g$. Then there exist complex numbers $\Omega_f^\pm$ and $\Omega_g^\pm$ such that 
\[ 
\frac{j!\tau(\chi)}{(-2\pi \sqrt{-1})^j} \cdot \frac{D^j}{\Omega_f^\pm} \cdot \binom{k-2}{j} \cdot L_{\Sigma_0}(f, \overline{\chi}, j+1) \equiv \cdot \frac{j!\tau(\chi)}{(-2\pi \sqrt{-1})^j} \cdot \frac{D^j}{\Omega_g^\pm} \cdot \binom{k-2}{j} \cdot L_{\Sigma_0}(g, \overline{\chi}, j+1)\,\bmod\,\P^r.
 \]
\end{proposition}
\begin{proof}
This follows from combining \eqref{eq:specialvalue} and \eqref{eq:congvarphi}. 
\end{proof}


\subsection{Mazur--Tate elements} \label{1.3}
Recall that $\mathcal{G}_n = \Gal(\QQ(\mu_{p^n})/\QQ)$. There is an isomorphism
$\mathcal{G}_n \to \left(\ZZ/p^n\ZZ\right)^\times$,  given by $\sigma_a \mapsto a$, where $\sigma_a$ is the automorphism $\zeta \mapsto \zeta^a$, for any $\zeta \in \mu_{p^n}$.

\begin{defn}
    \textup{
    For a modular symbol $\varphi \in H_c^1(\Gamma_0(N), L_{k-2}(R))$, we define the associated \emph{Mazur--Tate elements} $\vartheta_n(\varphi)$ of level $n \geq 1$ by 
    $$\vartheta_n(\varphi) = \sum_{a \in (\ZZ/p^n\ZZ)^\times} \varphi\,\bigg|\, \begin{pmatrix}
    1 & -a\\ 0 & p^n
\end{pmatrix} (\{\infty\} - \{0\})\cdot \sigma_a \in R[X,Y][\mathcal{G}_n].$$
    We write
    \[
    \vartheta_n(\varphi)=\sum_{j=0}^{k-2}\binom{k-2}{j}X^jY^{k-2-j}\vartheta_{n,j}(\varphi),
    \]
    where $\vartheta_{n,j}(\varphi)\in R[\cG_n]$.
    }
\end{defn}

We have the explicit formula
\[
\vartheta_{n,j}(\xi_f)=2\pi\sqrt{-1}\int_{-a/p^n}^\infty f(z)(p^nz+a)^j\mathrm{d}z\cdot\sigma_a\in\CC[\cG_n].
\]
Under the decomposition $\xi_f=\xi_f^++\xi_f^-$ as in Remark~\ref{CtoO}, we have
\[
\vartheta_{n,j}(\xi_f^\pm)=\pi\sqrt{-1}\left(\int_{-a/p^n}^\infty f(z)(p^nz+a)^j\mathrm{d}z \pm (-1)^{k+j} \int_{a/p^n}^\infty f(z)(p^nz-a)^j\mathrm{d}z\right)\cdot\sigma_a\in\CC[\cG_n].
\]
We define
\[
\theta_{n,j}(\xi_f^\pm):=\frac{\vartheta_{n,j}(\xi_f^\pm)}{\Omega_f^\pm}\in \cO[\cG_n].
\]
Recall that we can decompose $\mathcal{G}_{n+1}$ as:
$$\mathcal{G}_{n+1} \cong \Delta \times G_n$$
where $\Delta \cong (\ZZ/p\ZZ)^\times$ and $G_n$ is a cyclic group of order $p^n$. Recall furthermore that $\omega$ is  the Teichm\"{u}ller character on $\Delta$. When $R$ is a $\Zp$-algebra, we obtain an induced map $\omega^i: R[\mathcal{G}_{n+1}] \to R[G_n]$ for each $0 \leq i \leq p-2$.

\begin{defn}
\textup{
For integers $0\le i\le p-2$, $0\le j\le k-2$ and $n\ge0$, we define
$\theta_{n,j}(f,\omega^i)\in\cO[G_n]$ as the image of $\theta_{n+1,j}\left(\xi_f^{(-1)^{i}}\right)$ under $\omega^{i-j}$.
}
\end{defn}
More explicitly,
\begin{align*}
\theta_{n,j}(f,\omega^i)=\frac{\pi\sqrt{-1}}{\Omega_f^{(-1)^i}}\sum_{a\in(\ZZ/p^{n+1}\ZZ)^\times}\bigg(&\int_{-a/p^{n+1}}^\infty f(z)(p^{n+1}z+a)^j\rd z\\
    &+(-1)^{i+k+j}\int_{a/p^{n+1}}^\infty f(z)(p^{n+1}z-a)^j\rd z\bigg)\,\omega^{i-j}(a)\bar\sigma_a,
\end{align*}
where $\bar\sigma_a$ denotes the image of $\sigma_a$ under the natural projection $\cG_{n+1}\rightarrow G_n$. Moreover, if $\Up$ is a root of the Hecke polynomial $X^2-a_p(f)X+\epsilon_f(p)p^{k-1}$, we define the \emph{p-stabilized theta element} for $n\ge1$
\[
\theta_{n,j}(f,\Up,\omega^i)=
\frac{1}{\Up^{n+1}}\cdot\theta_{n,j}(f,\omega^i)-\frac{\epsilon_f(p)p^{k-2}}{\Up^{n+2}}\cdot\nu^n_{n-1}\theta_{n-1,j}(f,\omega^i)\in E(\Up)[G_{n}], 
\]
where $\nu^n_{n-1}:\cO[G_{n-1}]\rightarrow\cO[G_{n}]$ is the norm map that sends $\sigma\in G_{n-1}$ to the sum of the pre-images of $\sigma$ in $G_n$ under the projection map $\pi^n_{n-1} :G_n\rightarrow G_{n-1}$.

\noindent  For $n \geq 1$ and $0 \leq i \leq p-2$, we have
\begin{enumerate}
	\item{$\pi^{n+1}_n \theta_{n+1, j}(f, \omega^i) = a_p(f) \theta_{n,j}(f, \omega^i) - \epsilon_f(p)p^{k-2}\nu^n_{n-1} \theta_{n-1,j}(f, \omega^i)$,  \cite[(4.2)]{MTT};}
	\item{$\pi^{n+1}_n\theta_{n+1,j}(f,\Up,\omega^i)=\theta_{n,j}(f,\Up,\omega^i)$ \cite[\S10]{MTT}.}
\end{enumerate}


\section{Signed $p$-adic $L$-functions}\label{sec:2}

Recall that $E$ is a finite extension of $\Qp$ that contains all the Fourier coefficients of $f$ and $g$. The ring of integers of $E$ is denoted by $\cO$. We review the construction of signed $p$-adic $L$-functions. In \cite{LLZ0,LLZ0.5,BFSuper}, these elements are defined as the image of Kato's zeta elements under some local Coleman maps. We outline an alternative construction that relies on decomposing the Mazur--Tate elements using a logarithmic matrix.

While we only state our results for $f$ in this section, all results hold for $g$ analogously.

\subsection{Non-integral $p$-adic $L$-functions as limits of polynomials}

\begin{defn}
\textup{
Let $P(X)=\sum_{i=0}^\infty c_iX^i\in \cO\llbracket X \rrbracket\otimes E$. We define its norm as $\| P \|=\sup_i|c_i|_p$, where $|\cdot|_p$ is the $p$-adic norm, normalized by $|p|_p=p^{-1}$. Given a real number $\rho\le 1$, we define $\|P\|_\rho=\sup_{|z|_p<\rho}|P(z)|_p$.
}
\end{defn}	
We define $\Phi_n = \sum_{j=0}^{p-1} (1+X)^{jp^{n-1}}$, for $n \geq 1$, which is the $p^n$-th cyclotomic polynomial in $X+1$. Further, we define $\Phi_0 = X$. Let $\omega_n = (1+X)^{p^n} - 1$. Then we have the elementary identity $\Phi_n = \omega_n/\omega_{n-1}$ for $n\ge1$. Let us also recall the $p$-adic logarithm
\begin{align*}
	\log_p(1+X) &= X\prod_{n \geq 1} \frac{\Phi_n(1+X)}{p}.
\end{align*}
Recall that $\gamma$ is a topological generator of $G_\infty=\Gal(\QQ_\mathrm{cyc}/\QQ)$ and we identify $\cO\llbracket G_\infty\rrbracket$ with the power series ring $\cO\llbracket X\rrbracket$ by sending $\gamma$ to $1+X$. In what follows, $u$ denotes the image of $\gamma$ under the cyclotomic character $\chi_\cyc$. For an integer $h\ge 1$, we write 
\begin{align*}
\Phi_{n,h} :=  & \prod_{j=0}^{h-1} \Phi_n(u^{-j}(1+X)-1),\\
\omega_{n,h} := & \prod_{j=0}^{h-1} \omega_n(u^{-j}(1+X)-1),\\
\log_{p,h} :=  & \prod_{j=0}^{h-1} \log_p(u^{-j}(1+X)-1).
\end{align*}

Given an element $\sum_{n\ge 0}c_nX^n\in E\llbracket X\rrbracket$, we say that it is $O(\log^r)$ if $r$ is a real number such that
\[
\sup_n \frac{|c_n|_p}{n^r}<\infty.
\]
Similarly, we say that an element $\sum_{n\ge 0,\sigma\in\Delta}c_{n,\sigma}\sigma X^n\in E[\Delta]\llbracket X\rrbracket$ is $O(\log^r)$ if 
\[
\sup\frac{|c_{n,\sigma}|_p}{n^r}<\infty
\]
for all $\sigma\in\Delta$.

The following result comes from \cite[Lemmas~2.2 and 2.3, combined with Remark 2.4]{BFSuper}, which is based on \cite[proof of Lemme 1.2.2]{perrinriou94}. 

\begin{lemma}\label{lem:PRtwists}
Let $h\ge 1$ be an integer. For $ j \in\{0,\dots,h-1\}$, let $\left(Q_{n,j}\right)_{n\ge 0}$ be a sequence of polynomials in $E[X]$. 
Let $n\ge 0$, and let $d_n \in \mathbb{R}$ be a constant independent of $j$ such that
\[
 \left\| p^{-j(n+1)} \sum_{t=0}^j(-1)^{j-t}\binom{j}{t}Q_{n,t}\left(u^{-t}(1+X)-1\right)\right\| \le d_n.
\]
 Let $P_n\in E[X]$ be the unique polynomial of degree $<hp^n$ such that
\[
P_n\equiv Q_{n,j}(u^{-j}(1+X)-1)\,\bmod\,\omega_n(u^{-j}(1+X)-1) E[X]
\]
for all $0\le j\le h-1$ (the existence and uniqueness of $P_n$ follows from the Chinese remainder theorem). Then we have $\| P_n \| \le c_h d_n$, where $c_h$ is a constant that depends only on $h$. (If $h=1$, we may take $c_h=1$.)

Suppose that $Q_{n+1,j}\equiv Q_{n,j}\,\bmod\,\omega_nE[X]$ for all $n\ge0$. Then $P_{n+1}\equiv P_n\,\bmod\,\omega_{n,h}E[X]$. 
If, in addition, there exists a real number $r$ such that $0\le r <h$ and $d_n\le p^{rn}$ for all $n\ge0$, then the sequence $(P_n)_{n\ge0}$ converges to a power series $P_\infty\in E\llbracket X \rrbracket$ that is $O(\log^r)$.
\end{lemma}

We will show that the bound $c_h$ is in fact equal to 1 when $E/\Qp$ is unramified and $p>h$ by carefully analyzing \cite[proof of Lemme 1.2.2]{perrinriou94}. For $0 \leq \ell \leq h-1$, let
\[
\delta_\ell(X) = \sum_{t=0}^{\ell} (-1)^{\ell-t} \binom{\ell}{t} Q_{n,t}(u^{-t}(1+X) - 1). 
\]
One can verify that 
\begin{equation}\label{eq:def-H}
H(X,Z) := \sum_{\ell=0}^{h-1} \binom{Z}{\ell} \delta_\ell(X)    
\end{equation}
 is a polynomial of degree $<h$ in $Z$ over $\cO[X]$ satisfying $H(X,j) = Q_{n,j}(u^{-j}(1+X) - 1)$ for $0 \leq j \leq h-1$. Let 
\[ 
S(X,Y) = \sum_{j=0}^{h-1} Q_{n,j}(u^{-j}(1+X) - 1) \prod_{\substack{0 \leq t \leq h-1\\ t \neq j}} \frac{u^{-tp^n}(1 + Y) - 1}{u^{(j-t)p^n} - 1}. 
\]
It is the unique polynomial of degree $<h$ in $Y$ with coefficients in $\mathcal{O}[X]$, satisfying $S(X, u^{jp^n} - 1) = Q_{n,j}(u^{-j}(1+X) - 1)$. 
Furthermore, the polynomial $P_n(X)$ in the statement of Lemma~\ref{lem:PRtwists} is equal to $S(X, (1+X)^{p^n} - 1)$. 

In what follows, given a two-variable polynomial $R(X,Y)\in\cO[X,Y]$ and a real number $r$, we write
\[
\|R(X,Y)\|_r:=\sup\left\{|R(x,y)|_{p}:|x|_p\le 1,|y|_p\le r\right\}.
\]

Let $\rho$ be a real number satisfying $p^{-1}<\rho < p^{-1/(p-1)}$. Note that $\rho/|u^{p^n}-1|_p=\rho p^{n+1}>1$ for all $n\ge0$. Consider the homeomorphism between the following open disks
\begin{align*}
\left\{ y:|y|_p < \rho\right\} &\rightarrow \left\{z:|z|_p < \rho/|u^{p^n} - 1|_p\right\}\\
y&\mapsto \log(1+y)/\log u^{p^n}\\
u^{zp^n} - 1&\mapsfrom z.
\end{align*}
It gives the following equality of sup-norms:
\begin{equation}\label{avprop4.4}
	\| S(X,Y) \|_\rho = \| H(X, Z) \|_{\rho/|u^{p^n} - 1|_p}.
\end{equation}
It follows from the expression \eqref{eq:def-H} that
\[
\| H(X,Z) \|\strut_{\tfrac{\rho}{|u^{p^n} - 1|_p}} 
= \sup_{0 \leq j \leq h-1} \| \delta_j(X) \| \left\| \binom{Z}{j} \right\|_{ \tfrac{\rho}{|u^{p^n} - 1|_p}}.
\]
Assuming $p >h-1$, we have
\[
\left\| \binom{Z}{j} \right\|_{ \tfrac{\rho}{|u^{p^n} - 1|_p}}
	= \left\| \frac{Z(Z-1) \hdots (Z - j + 1)}{j!} \right\|_{\tfrac{\rho}{|u^{p^n} - 1|_p}}
	= \left(\frac{\rho}{|u^{p^n} - 1|_p}\right)^j
\]
since $\rho/|u^{p^n} - 1|_p > 1$. 

Let us write $S(X,Y) = \sum_{\ell=0}^{h-1} S_\ell(X) Y^\ell$.
If $\sup_j \| \delta_j(X) \| \leq p^{-(n+1)j}$, then equation \eqref{avprop4.4} becomes
\[ \sup_{0 \leq \ell \leq h-1} (\| S_\ell(X) \| \rho^\ell) 
= \sup_{0 \leq j \leq h-1} \left( \| \delta_j(X) \| \cdot \frac{\rho^j}{|u^{p^n} - 1|_p^j} \right)
= \sup_{0 \leq j \leq h-1} \left( \| p^{-(n+1)j} \cdot \delta_j(X) \| \cdot \rho^j \right)
\leq \sup_{0 \leq j \leq h-1} \rho^j
= 1. \]
Therefore, $\| S_\ell(X) \| \leq \rho^{-\ell}$ for each $0 \leq \ell \leq h-1$, which implies that
\[
\| P_n(X) \| 
	= \left\| \sum_{\ell=0}^{h-1} S_\ell(X)((1+X)^{p^n}-1)^\ell \right\|
	\leq \rho^{-(h-1)} .
\]
Since $\rho$ can be arbitrarily close to $p^{-1/(p-1)}$, we have
\[
\|P_n(X)\|\le p^{\frac{h-1}{p-1}}<p
\]
if we suppose furthermore that $p>h$. If in addition $P_n(X)$ is defined over an unramified extension of $\Qp$, then $\|P_n(X)\|$ is an integral power of $p$. In particular, $\|P_n(X)\|\le 1$.
We can summarize the discussion above with the following refinement of Lemma~\ref{lem:PRtwists}.

\begin{proposition}\label{prop:Pn-integral}
Let $h$ and $Q_{n,j}$ be as in the statement of Lemma~\ref{lem:PRtwists}. Assume that $p > h$ and that $E/\Qp$ is unramified. Then $c_h$ can be taken to $1$. In particular, if $d_n=1$, we have $P_n\in\cO[X]$.
\end{proposition}

In what follows, we write $E'=E(\alpha_f)=E(\beta_f)$, where $\alpha_f$ and $\beta_f$ are the roots of $X^2-a_p(f)X+\epsilon_f(p)p^{k-1}$. We write $\cO'$ for the ring of integers of $E'$. For simplicity, we shall write $\Upsilon=\Upsilon_f$ for $\Upsilon\in\{\alpha,\beta\}$.

There is a natural identification
\begin{equation}\label{eq:identification}
E[G_n] = E\llbracket \gamma-1 \rrbracket / (\gamma^{p^n}-1) = E\llbracket X \rrbracket/(\omega_n).
\end{equation}
The same is true if we replace $E$ by $E'$.
Note that $G_n$ is a cyclic group of order $p^n$ generated by the image of $\gamma$ in $G_n$.
Let $a\in(\ZZ/p^n\ZZ)^\times$. The element $\bar\sigma_a\in E[G_n]$ (or $E'[G_n]$) can be regarded as the polynomial $(1+X)^{m(a)}$, where $m(a)$ is the unique integer such that $0\le m(a)\le p^n-1$ and $\bar\sigma_a=\gamma^{m(a)}\mod \gamma^{p^n}$.

We define $Q_{n,j}(f,\omega^i)\in\cO[X]$ as the image of the theta element $\theta_{n,j}(f,\omega^{i})$ under this identification. We define $Q_{n,j}(f,\Upsilon,\omega^i)\in E'[X]$ similarly.

\begin{lemma}\label{lem:verify-PR}
Suppose that $p>k-1$. For an integer $j$ such that $0\le j\le k-2$, we have  
\[ \left\| \sum_{t=0}^j (-1)^{j-t}\binom{j}{t}Q_{n,t}(f,\omega^i)(u^{-t}(1+X)-1)\right\| \le p^{-(n+1)j}.\]
\end{lemma}
\begin{proof}
    For $t\in\{0,\dots,k-2\}$, we have
\begin{align*}
&Q_{n,t}(f,\omega^i)(u^{-t}(1+X)-1)\\
=\ &\frac{\pi\sqrt{-1}}{\Omega_f^{(-1)^i}}\sum_{a}\left(\int^\infty_{-a/p^{n+1}} f(z)(p^{n+1}z+a)^t\rd z+(-1)^{i+k+t}\int_{a/p^{n+1}}^\infty f(z)(p^{n+1}z-a)^t\rd z\right) u^{-tm(a)}\omega^{i-t}(a)(1+X)^{m(a)}\\
=\ &    \frac{\pi \sqrt{-1}}{\Omega_f^{(-1)^i}}\sum_{a}\left(\int^\infty_{-a/p^{n+1}}f(z)\left(\frac{p^{n+1}z+a}{\omega(a)u^{m(a)}}\right)^t\rd z+(-1)^{i+k+t}\int_{a/p^{n+1}}^\infty f(z)\left(\frac{p^{n+1}z-a}{\omega(a)u^{m(a)}}\right)^t\rd z\right)\omega^{i}(a) (1+X)^{m(a)}.
\end{align*}
Let $\hat a=\omega(a)u^{m(a)}$. Note that $\hat{a}\equiv a\mod p^{n+1}\Zp$. Hence,
\begin{align*}
& \sum_{t=0}^j (-1)^{j-t}\binom{j}{t}Q_{n,t}(f,\omega^i)(u^{-t}(1+X)-1)\\   
=\ &    \frac{\pi \sqrt{-1}}{\Omega_f^{(-1)^i}}\sum_{a}\left(\int_{-a/p^{n+1}}^\infty f(z)\left(\frac{p^{n+1}z+a}{\hat{a}}-1\right)^j\rd z+(-1)^{i+k+j}\int_{a/p^{n+1}}^\infty f(z)\left(\frac{p^{n+1}z-a}{\hat{a}}+1\right)^j\rd z\right)\omega^{i}(a) (1+X)^{m(a)}\\
=\ &   \frac{\pi \sqrt{-1}}{\Omega_f^{(-1)^i}}\sum_{a}\left(\int^\infty_{-a/p^{n+1}}f(z)z^j\rd z+(-1)^{i+k}\int_{a/p^{n+1}}^\infty f(z)(-z)^j\rd z\right)\frac{p^{(n+1)j}}{\hat a^{j}}\omega^{i}(a) (1+X)^{m(a)},
\end{align*}
which belongs to $p^{(n+1)j} \cdot \cO[X]$ since
\[
\xi_f^{(-1)^i}([\infty]-[a/p^{n+1}])=\pi\sqrt{-1}\left(\int^\infty_{-a/p^{n+1}}f(z)(zX+Y)^{k-2}\rd z+(-1)^{i+k}\int_{a/p^{n+1}}^\infty f(z)(-zX+Y)^{k-2}\rd z\right)
\]
belongs to $\Omega_f^{(-1)^i} \cdot \cO[X,Y]$. Hence, $$ \frac{\pi \sqrt{-1}}{\Omega_f^{(-1)^i}}\left(\int^\infty_{-a/p^{n+1}}f(z)z^j\rd z+(-1)^{i+k}\int_{a/p^{n+1}}^\infty f(z)(-z)^j\rd z\right)\in\cO,$$ from which the lemma follows.
\end{proof}

Let $P_{n}(f,\omega^i)$ be the unique polynomial of degree $<(k-1)p^n$ such that
\[
P_n(f,\omega^i)\equiv Q_{n,j}(f,\omega^i)(u^{-j}(1+X)-1)\,\bmod\,\omega_n(u^{-j}(1+X)-1)E[X]
\]
for $0\le j\le k-2$. 

\begin{corollary}\label{cor:integral-Pn}
Suppose $p > k-1$ and that $E/\Qp$ is unramified. For all $n\ge0$ and $0\le i\le p-2$, we have $P_{n}(f,\omega^i)\in\cO[X]$.
\end{corollary}
\begin{proof}
This follows from applying Lemma~\ref{lem:PRtwists} to $Q_{n,j}(f,\omega^i)$ with $h=k-1$ and $d_n = 1$ coming from  Lemma~\ref{lem:verify-PR}. The last assertion is a consequence of Proposition~\ref{prop:Pn-integral}.
\end{proof}

We now turn our attention to the $p$-stabilized version of these polynomials. 
For each $\Up\in\{\alpha,\beta\}$, define $P_{n}(f,\Up,\omega^i)$ as the unique polynomial of degree $<(k-1)p^n$ such that
\[
P_{n}(f,\Up,\omega^i)\equiv Q_{n,j}(f,\Up,\omega^i)(u^{-j}(1+X)-1)\mod\omega_n(u^{-j}(1+X)-1)E'[X]
\]
for $0\le j\le k-2$.

\begin{lemma}\label{lem:Qnj-lambda}
For all $n\ge1$ and $j=0,\dots,k-2$, we have
\[
    Q_{n,j}(f,\Upsilon,\omega^i)=\frac{1}{\Upsilon^{n+1}}Q_{n,j}(f,\omega^i)-\frac{\epsilon(p) p^{k-2}}{\Upsilon^{n+2}}\Phi_{n}\cdot Q_{n-1,j}(f,\omega^i).
\]
\end{lemma}
\begin{proof}
 This follows from the three-term relation satisfied by the Mazur--Tate elements stated at the end of \S\ref{1.3}. The map $\nu_{n/n-1}$ sends $(1+X)^t$ to $\sum_{s=0}^{p-1} (1+X)^{t+sp^{n-1}}$. Thus, it is equivalent to the multiplication by $\Phi_n$.
\end{proof}

\begin{lemma}\label{lem:Phi-n}
    There exists a polynomial $\tilde\Phi_{n,k-1}$ such that for all $0\le j\le k-2$, we have
    \[
    \tilde\Phi_{n,k-1}\equiv \Phi_n(u^{-j}(1+X)-1)\mod \omega_n(u^{-j}(1+X)-1)
    \]
    with $\| p^{k-2}\tilde\Phi_{n,k-1} \| \le c_{k-1}$. If $p\ge k-1$, then $p^{k-2}\tilde\Phi_{n,k-1}\in\Zp[X]$.
\end{lemma}
\begin{proof}
    Let $0\le j\le k-2$. We have
    \begin{align*}
&    \sum_{t=0}^j(-1)^{j-t}\binom{j}{t}\Phi_n(u^{-t}(1+X)-1)\\    
=\ & \sum_{t=0}^j(-1)^{j-t}\binom{j}{t}\sum_{s=0}^{p-1}(u^{-t}(1+X))^{sp^{n-1}}\\
=\ &\sum_{s=0}^{p-1}(u^{-sp^{n-1}}-1)^j(1+X)^{sp^{n-1}}\in p^{nj}\Zp[X]
    \end{align*}
    since $u\in 1+p\Zp$. Thus, the existence of $\tilde\Phi_{n,k-1}$ follows from Lemma~\ref{lem:PRtwists}. The last assertion is a consequence of Proposition~\ref{prop:Pn-integral}.
\end{proof}

\begin{lemma}\label{lem:Pn-lambda}
      For all $n\ge1$, $0\le i\le p-2$ and $\Up\in\{\alpha,\beta\}$, we have
    \[
    P_n(f,\Up,\omega^i)\equiv\frac{1}{\Up^{n+1}}P_n(f,\omega^i)-\frac{\epsilon(p)p^{k-2}}{\Up^{n+2}}\tilde\Phi_{n,k-1}\cdot P_{n-1}(f,\omega^i)\mod\omega_{n,k-1}.
    \]
\end{lemma}
\begin{proof}
    This follows from combining Lemmas~\ref{lem:Qnj-lambda} and~\ref{lem:Phi-n}.
\end{proof}

\begin{corollary}\label{cor:integral-Pn-lambda}
Suppose $p>k-1$ and that $E/\Qp$ is unramified. For all $n\ge0$ and $0\le i\le p-2$, we have
\[ \left\| P_{n}(f,\Up,\omega^i) \right\| \le  \left|\frac{1}{\Up^{n+2}}\right|_p . \]
\end{corollary}
\begin{proof}
    This follows from combining Corollary~\ref{cor:integral-Pn} with Lemmas~\ref{lem:Phi-n} and \ref{lem:Pn-lambda}.
\end{proof}

In particular, $\| p^{\ord_p(\Up)n}P_{n}(f,\Up,\omega^i) \| = O(1)$. Combined with Lemma~\ref{lem:verify-PR}, we can apply Lemma~\ref{lem:PRtwists} after setting $h=k-1$, $r=\ord_p(\Up)$ (which is $<k-1$ as $f$ is non-ordinary at $p$), and $Q_{n,j}=Q_{n,j}(f,\Up,\omega^i)$. In particular, we can define $L_{p}(f,\Up,\omega^i,X)\in E'\llbracket X\rrbracket$ as the limit of the polynomials $P_n(f,\Up,\omega^i)$.
In other words,
\begin{equation}\label{eq:limit-Lp} 
L_{p}(f,\Up,\omega^i,X) = \lim_{n \to \infty} P_n(f,\Up,\omega^i),
\end{equation}
which is $O(\log^{\ord_p(\Up)})$. Furthermore,
\begin{equation}
\label{eq:padicL-cong}
L_{p}(f,\Up,\omega^i,X)\equiv P_n(f,\Up,\omega^i)\,\bmod\,\omega_{n,k-1}.    
\end{equation}

Let $L_p(f,\Up,X)\in E'[\Delta]\llbracket X \rrbracket$
denote the Amice transform of the $E'$-valued distribution constructed in \cite{amicevelu75,visik76,MTT}. 
The limit $L_{p}(f,\Up,\omega^i,X)$ given by \eqref{eq:limit-Lp} is the power series obtained from $L_p(f,\Up,X)$ after applying $\omega^i$ to $\Delta$.


\subsection{Construction of signed $p$-adic $L$-functions}

From now on, we assume that Assumption~\ref{asi:FL} holds. In particular, this provides an explicit basis for the Wach module attached to $T_f$, which in turn allows us to define a logarithmic matrix, as studied in \cite{BFSuper}. We write $\cO\llbracket\pi\rrbracket$ for the ring of power series in $\pi$, which is equipped with an $\cO$-linear operator $\varphi$ that sends $\pi$ to $(1+\pi)^p-1$ and an $\cO$-linear action by $\cG_\infty$ given by $\sigma\cdot\pi=(1+\pi)^{\chi_\cyc(\sigma)}-1$. The Mellin transform that sends $a\in\cO\llbracket\cG_\infty\rrbracket$ to $a\cdot (1+\pi)$ induces an isomorphism
\[
\fM:\cO\llbracket\cG_\infty \rrbracket\stackrel{\sim}{\longrightarrow}\cO\llbracket\pi\rrbracket^{\psi=0},\] where $\psi$ is a left-inverse of $\vp$.

\begin{defn}
Let $q=\vp(\pi)/\pi \in\cO\llbracket \pi \rrbracket$ and $\delta=p/(q-\pi^{p-1})\in\cO\llbracket \pi \rrbracket^\times$. We define 
\[
P_f=\begin{bmatrix}
    0&\displaystyle\frac{-1}{\epsilon_f(p)q^{k-1}}\\ \delta^{k-1}&\displaystyle\frac{a_p(f)}{\epsilon_f(p)q^{k-1}}
\end{bmatrix}.
\]
For $n\ge1$, we define $C_{n,f}$ to be the $2\times2$ matrix of polynomials of degree $<(k-1)p^n$ that coincide with the image of 
\[
\fM^{-1}\left((1+\pi)\vp^n(P_f^{-1})\cdots \vp(P_f^{-1})\right)\,\bmod\,\omega_{n,k-1}.
\]

We define
\[
A_f=\begin{bmatrix}
    0&\displaystyle\frac{-1}{\epsilon_f(p)p^{k-1}}\\ 1&\displaystyle\frac{a_p(f)}{\epsilon_f(p)p^{k-1}}
\end{bmatrix},\quad Q_f=\begin{bmatrix}
    \alpha &-\beta\\ -\alpha\beta&\alpha\beta
\end{bmatrix}.
\]
\end{defn}

\begin{proposition}\label{prop:log-matrix}
The determinant of $C_{n,f}$ is equal, up to a unit of $\cO\llbracket X \rrbracket^\times$, to $\omega_{n,k-1}/\Phi_{0,k-1}$.
    The sequence of matrices $A_f^{n+1}C_{n,f}$ converges to a matrix $M_{\log,f}$ defined over $E\llbracket X \rrbracket$. Furthermore, the entries in the first row of $Q_f^{-1}M_{\log,f}$ are elements of $E'\llbracket X\rrbracket$ that are $O(\log^{\ord_p(\alpha)})$, while those in the second row are $O(\log^{\ord_p(\beta)})$.
\end{proposition}
\begin{proof}
    See \cite[Lemmas 2.6-2.8]{BFSuper}.
\end{proof}

\begin{theorem}\label{thm:decomp}\label{lem:signed-cong}
  For all $0\le i\le p-2$, there exist $L_p(f,\sharp,\omega^i,X),L_p(f,\flat,\omega^i,X)\in \cO\llbracket X \rrbracket\otimes_{\cO}E$ such that 
\[
\frac{1}{\alpha-\beta}\cdot \begin{bmatrix}
   L_p(f,\alpha,\omega^i,X)\\ L_p(f,\beta,\omega^i,X)
\end{bmatrix}=
 Q_f^{-1}M_{\log,f}\begin{bmatrix}
    L_p(f,\sharp,\omega^i,X)\\ L_p(f,\flat,\omega^i,X)
\end{bmatrix}.
\]
Furthermore, for all $n\ge1$, we have
\[\begin{bmatrix}
  P_n(f,\omega^i) \\ -\epsilon(p)p^{k-2}\tilde\Phi_{n,k-1}P_{n-1}(f,\omega^i)
\end{bmatrix}
\equiv C_{n,f}\begin{bmatrix}
    L_p(f,\sharp,\omega^i,X)\\ L_p(f,\flat,\omega^i,X)\end{bmatrix}\,\bmod\,\omega_{n,k-1}.    
\]

\end{theorem}
\begin{proof}
Let $F_\Up=\frac{1}{\alpha-\beta}\cdot L_p(f,\Up,\omega^i,X)$. We can apply \cite[Proposition~2.11]{BFSuper} to the pair of elements $(F_\alpha,F_\beta)$ to deduce that there exist $L_p(f,\sharp,\omega^i,X), L_p(f,\flat,\omega^i,X)\in \cO\llbracket X\rrbracket\otimes_\cO E'$ such that
\[
\frac{1}{\alpha-\beta}\cdot \begin{bmatrix}
   L_p(f,\alpha,\omega^i,X)\\ L_p(f,\beta,\omega^i,X)
\end{bmatrix}=
 Q_f^{-1}M_{\log,f}\begin{bmatrix}
    L_p(f,\sharp,\omega^i,X)\\ L_p(f,\flat,\omega^i,X)
\end{bmatrix}.
\]

    A direct computation shows that
    \[
    \frac{1}{\alpha-\beta}\cdot A_f^{-n-1}Q_f=\frac{1}{\alpha-\beta}\cdot Q_f\begin{bmatrix}
        \alpha^{n+1}&0\\0&\beta^{n+1}
    \end{bmatrix}=\frac{1}{\alpha-\beta}\cdot\begin{bmatrix}
        \alpha^{n+2}&-\beta^{n+2}\\ -\alpha^{n+2}\beta&\alpha\beta^{n+2}
    \end{bmatrix}.
    \]
    Therefore, combined with \eqref{eq:padicL-cong} and Lemma~\ref{lem:Pn-lambda}, we deduce that
    \begin{align*}
     &   \frac{1}{\alpha-\beta}\cdot A_f^{-n-1}Q_f\begin{bmatrix}
   L_p(f,\alpha,\omega^i,X)\\ L_p(f,\beta,\omega^i,X)
\end{bmatrix}\\
\equiv\ &\frac{1}{\alpha-\beta}\cdot\begin{bmatrix}
        \alpha^{n+2}&-\beta^{n+2}\\ -\alpha^{n+2}\beta&\alpha\beta^{n+2}
    \end{bmatrix}\begin{bmatrix}
        \alpha^{-n-1}&\alpha^{-n-2}\\ \beta^{-n-1}&\beta^{-n-2}
    \end{bmatrix}\begin{bmatrix}
  P_n(f,\omega^i) \\ -\epsilon_f(p)p^{k-2}\tilde\Phi_{n,k-1}P_{n-1}(f,\omega^i)
\end{bmatrix}\,\bmod\,\omega_{n,k-1}\\
\equiv\ & \begin{bmatrix}
  P_n(f,\omega^i) \\ -\epsilon_f(p)p^{k-2}\tilde\Phi_{n,k-1}P_{n-1}(f,\omega^i)
\end{bmatrix}\,\bmod\,\omega_{n,k-1}.
    \end{align*}
Thus, it follows from Proposition~\ref{prop:log-matrix} that
\[\begin{bmatrix}
  P_n(f,\omega^i) \\ -\epsilon(p)p^{k-2}\tilde\Phi_{n,k-1}P_{n-1}(f,\omega^i)
\end{bmatrix}\equiv   A_f^{-n-1}Q_f Q_f^{-1}M_{\log,f}\begin{bmatrix}
    L_p(f,\sharp,\omega^i,X)\\ L_p(f,\flat,\omega^i,X)
\end{bmatrix}
\equiv C_{n,f}\begin{bmatrix}
    L_p(f,\sharp,\omega^i,X)\\ L_p(f,\flat,\omega^i,X)\end{bmatrix}\,\bmod\,\omega_{n,k-1}.    
\]

Since the left-hand side consists of elements belonging to $\P^{-s}\cO[X]$ for some integer $s$ that is independent of $n$, we deduce that
\[
\begin{bmatrix}
  \P^s  L_p(f,\sharp,\omega^i,X)\\\P^s L_p(f,\flat,\omega^i,X)\end{bmatrix}\in\varprojlim \cO\llbracket X\rrbracket^{\oplus 2}/\ker h_n=\cO\llbracket X\rrbracket^{\oplus 2},
\]
where $h_n$ denotes the map $ \cO\llbracket X\rrbracket^{\oplus 2}\rightarrow  \cO\llbracket X\rrbracket^{\oplus 2}/\omega_{n,k-1}$ given by the multiplication by $C_{n,f}$ and the last equality is given by \cite[Lemma~2.9]{BFSuper}. Thus,  $L_p(f,\sharp,\omega^i,X)$ and $L_p(f,\flat,\omega^i,X)$ belong to $\cO\llbracket X\rrbracket\otimes_\cO E$ (rather than $\cO\llbracket X\rrbracket\otimes_\cO E'$). 
\end{proof}

\begin{theorem}\label{thm:Lp-integral}
Suppose $p>k-1$ and that $E/\Qp$ is unramified. Then
\begin{align*}
L_p(f,\sharp,\omega^i,X),\ L_p(f,\flat,\omega^i,X) &\in \cO\llbracket X \rrbracket.
\end{align*}
\end{theorem}

\begin{proof}
We consider two separate cases, either $\| L_p(f,\sharp,\omega^i,X) \| \ge \| L_p(f,\flat,\omega^i,X) \| $, or $\| L_p(f,\sharp,\omega^i,X) \| \le \| L_p(f,\flat,\omega^i,X) \| $.

We first consider the case where $\| L_p(f,\sharp,\omega^i,X) \| \ge \| L_p(f,\flat,\omega^i,X) \|$ and suppose for contradiction that $\| L_p(f,\sharp,\omega^i,X) \| > 1$.
Let $L_p(f,\sharp,\omega^i)_n$ and $L_p(f,\flat,\omega^i)_n$ be the images of $L_p(f,\sharp,\omega^i,X)$ and $L_p(f,\flat,\omega^i,X)$ modulo $\omega_{n,k-1}$, respectively. When $n$ is sufficiently large, we have $\| L_p(f,\sharp,\omega^i,X) \| = \| L_p(f,\sharp,\omega^i)_n \|$ and $\| L_p(f,\flat,\omega^i,X) \| = \| L_p(f,\flat,\omega^i)_n \|$.

 From the definition of $P_f$ and \cite[Lemma~2.6]{BFSuper} it follows that
\[
C_{n,f}\equiv \begin{bmatrix}
    0&*\\
    *\Phi_{n,k-1}&0
\end{bmatrix}\cdots \begin{bmatrix}
    0&*\\
    *\Phi_{1,k-1}&0
\end{bmatrix}\mod\P,
\]
where $*$ represents a unit in $\cO_f\llbracket X \rrbracket$. Let $\Phi_{n,k-1}^+=\displaystyle\prod_{m\le n,\text{even}}\Phi_{m,k-1}$ and $\Phi_{n,k-1}^-=\displaystyle\prod_{m\le n,\text{odd}}\Phi_{m,k-1}$. If $n$ is odd,
\[
C_{n,f}\equiv \begin{bmatrix}
    0&*\Phi_{n,k-1}^+\\
    *\Phi_{n,k-1}^-&0
\end{bmatrix}\mod\P,
\]
whereas when $n$ is even
\[
C_{n,f}\equiv \begin{bmatrix}
    *\Phi_{n,k-1}^-&0\\
    0&*\Phi_{n,k-1}^+
\end{bmatrix}\mod\P.
\]
When $n$ is even and sufficiently large, we have
\[
C_{n,f}\begin{bmatrix}
    L_p(f,\sharp,\omega^i)_n\\L_{p}(f,\flat,\omega^i)_n
\end{bmatrix}=\begin{bmatrix}
   (*\Phi_{n,k-1}^-+\P \Theta)L_p(f,\sharp,\omega^i)_n+ \P\Xi L_p(f,\flat,\omega^i)_n\\ \Psi
\end{bmatrix}
\]
for some $\Theta,\Xi\in \cO[X]$ and $\Psi\in E[X]$. Since $\Phi_{n,k-1}^-$ is monic, we have
\[
\|\P\Xi L_p(f,\flat,\omega^i)_n\| < \| L_p(f,\flat,\omega^i,X) \| \le \| L_p(f,\sharp,\omega^i,X) \| = \|(*\Phi_{n,k-1}^-+\P \Theta)L_p(f,\sharp,\omega^i)_n\|.
\]
Therefore, the strong triangle inequality implies that
\[
  \|(*\Phi_{n,k-1}^-+\P \Theta)L_p(f,\sharp,\omega^i)_n+ \P\Xi L_p(f,\flat,\omega^i)_n\| = \|(*\Phi_{n,k-1}^-+\P \Theta)L_p(f,\sharp,\omega^i)_n\|=\|L_p(f,\sharp,\omega^i,X) \| > 1.
\]
Combined with Theorem~\ref{thm:decomp}, we deduce 
$\|P_n(f,\omega^i)\| > 1$. This contradicts Corollary~\ref{cor:integral-Pn}.

If $\| L_p(f,\flat,\omega^i,X) \| \ge \| L_p(f,\sharp,\omega^i,X) \| $, we suppose $\| L_p(f,\flat,\omega^i,X) \|>1$ and  deduce a similar contradiction by considering an odd $n$ that is sufficiently large. Therefore, we conclude that in both cases, we have $\| L_p(f,\sharp,\omega^i,X) \|, \|L_p(f,\flat,\omega^i,X) \| \le 1$, as required.
\end{proof}

\section{Congruence between signed $p$-adic $L$-functions}\label{sec:3}

Let $f$ and $g$ be the two normalized cuspidal eigenforms defined in the introduction. 
For a prime $v$, we define 
\[
\cE_v(f)=\cP_v(f,v^{-1}\gamma_v)\in \cO\llbracket \cG_\infty\rrbracket=\cO[\Delta]\llbracket G_\infty\rrbracket,
\]
where $\cP_v(f,X)\in\cO[X]$ is defined as in \eqref{eq:Euler} and $\gamma_v\in\cG_\infty$ sends $\zeta\in\mu_{p^{\infty}}$ to $\zeta^v$. For $i\in\{0,\dots,p-2\}$, we define
$\mathcal{E}_v(f, \omega^i)\in\cO\llbracket X\rrbracket$ to be the image of $\cE_v(f)$ after applying $\omega^i$ to $\Delta$. 

If $\Sigma$ is a finite set of primes, let $Q_{\Sigma,n}(f, \omega^i) = Q_n(f, \omega^i) \times \prod_{v \in \Sigma} \mathcal{E}_{v}(f, \omega^i)\mod\omega_n$ be the $\Sigma$-imprimitive version of the polynomial $Q_n$ introduced in Section~\ref{sec:2}. 

\begin{remark}\label{rmk:imp-dep}
\textup{
Take $\Sigma = \Sigma_0$, the finite set of primes dividing $N_f N_g$. We recall that $Q_n(f,j)$ is the image of $\theta_{n,j}(f,\omega^i)$ under the identification \eqref{eq:identification}. We also recall that we chose $u_f^\pm, u_g^\pm = 1$ in the proof of Proposition~\ref{prop:L-cong} so that $\Omega_f^\pm = \Omega_F^\pm$ and $\Omega_g^\pm = \Omega_G^\pm$. We can therefore infer that $Q_{\Sigma_0,n}(f,\omega^i)$ and $Q_n(F,\omega^i)$ interpolate the same $L$-values, so they are congruent to each other modulo $\omega_n$. By the uniqueness of the Chinese remainder theorem, we have $P_{\Sigma_0,n}(f,\omega^i) = P_n(F,\omega^i)$.
}
\end{remark}

From this point onward, we assume that Assumptions~\ref{asi:FL} and~\ref{asi:unram} hold. As $E/\Qp$ is unramified, we can take $\P=p$.

\begin{proposition}\label{prop:cong-Pn}
Let $f$ and $g$ be modular forms of weight $k \geq 2$ satisfying Assumptions~\ref{asi:main}-\ref{asi:unram}. Then for all $n \geq 1$ and $0 \leq i \leq p-2$, we have
\[ P_{\Sigma_0,n}(f,\omega^i) \equiv P_{\Sigma_0,n}(g,\omega^i) \,\bmod\,p^r \cdot \omega_{n,k-1} \cdot \mathcal{O}[X]. \]
\end{proposition}
\begin{proof}
We have shown in \eqref{eq:congvarphi} that for each $n \geq 1$,
\[ \sum_{a \in (\ZZ/p^{n+1}\ZZ)} \varphi_F \bigg| \begin{pmatrix} 1 & -a\\ 0 & p^n \end{pmatrix} ( \{\infty\} - \{0\} ) \equiv \sum_{a \in (\ZZ/p^{n+1}\ZZ)} \varphi_G \bigg| \begin{pmatrix} 1 & -a\\ 0 & p^n \end{pmatrix} ( \{\infty\} - \{0\} )\,\bmod\,p^r. \]

If we set $Q_{n,j}$ as $Q_{n,j}(F,\omega^i)-Q_{n,j}(G,\omega^i)$ and replace the modular form $f$ in the proof of Lemma~\ref{lem:verify-PR} by $F-G$, we deduce that 
\[ \left\| \sum_{t=0}^j (-1)^{j-t}\binom{j}{t}Q_{n,t}(u^{-t}(1+X)-1)\right\| \le p^{-(n+1)j-r}.\]
Therefore, we can apply Proposition~\ref{prop:Pn-integral} with $d_n=p^{-r}$ to deduce that  $\| P_n(F, \omega^i) - P_n(G, \omega^i) \| \leq p^{-r}$. Hence, the proposition follows from Remark~\ref{rmk:imp-dep}.
\end{proof}

\begin{lemma}\label{lem:eps-cong}
If $f$ and $g$ are $p^r$-congruent modular forms, then $\epsilon_f(p) \equiv \epsilon_g(p)\,\bmod\,p^r$.
\end{lemma}
\begin{proof}
The hypothesis that $f$ and $g$ are $p^r$-congruent implies that
\[
a_{\ell^2}(f)\equiv a_\ell(f)^2-\epsilon_f(\ell)\ell^{k-1}\equiv a_{\ell^2}(g)\equiv a_\ell(g)^2-\epsilon_g(\ell)\ell^{k-1}\,\bmod\,p^r
\]
for all $\ell\nmid N_fN_g$. Therefore, for all $\ell\nmid pN_fN_g$, we have
\[
\epsilon_f(\ell)\equiv\epsilon_g(\ell)\,\bmod\,p^r.
\]
By Dirichlet's theorem, there exists a prime number $\ell\ne p$ that is congruent to $p\mod N_fN_g$. For any such $\ell$, we have $\epsilon_f(\ell)=\epsilon_f(p)$ and $\epsilon_g(\ell)=\epsilon_g(p)$. Hence, the lemma follows.
\end{proof}

\begin{lemma}\label{lem:cong-Cn}
If $f$ and $g$ are $p^r$-congruent modular forms, then we have for $n \geq 1$, an entry-wise congruence
\[ C_{n,f} \equiv C_{n,g}\,\bmod\,p^r \cdot \omega_{n,k-1}. \]
\end{lemma}
\begin{proof}
As $a_p(f)\equiv a_p(g),\epsilon_f(p)\equiv\epsilon_g(p)
\mod p^r$, the lemma follows immediately from the definition of $C_{n,f}$ and $C_{n,g}$.
\end{proof}

We define for $\natural \in \{ \sharp, \flat \}$, the $\Sigma_0$-imprimitive signed $p$-adic $L$-function 
\begin{equation}\label{eq:Lp-imp}
L_{p,\Sigma_0}(f,\natural,\omega^i,X) := L_p(f,\natural,\omega^i,X) \times \prod_{v \in \Sigma_0} \mathcal{E}_v(f,\omega^i).
\end{equation}

\begin{theorem}\label{thm:cong-Lp}
Let $f$ and $g$ be $p^r$-congruent modular forms of weight $k \geq 2$ satisfying Assumptions~\ref{asi:main}-\ref{asi:unram}. Then for $0 \leq i \leq p-2$, we have a pair of congruences
\begin{align*}
    L_{p,\Sigma_0}(f,\sharp,\omega^i,X) &\equiv L_{p,\Sigma_0}(g,\sharp,\omega^i,X) \,\bmod\,p^{r} \cdot \mathcal{O}\llbracket X \rrbracket,\\
    L_{p,\Sigma_0}(f,\flat,\omega^i,X) &\equiv L_{p,\Sigma_0}(g,\flat,\omega^i,X) \,\bmod\,p^{r} \cdot \mathcal{O}\llbracket X \rrbracket.
\end{align*}
\end{theorem}
\begin{proof}
The proof resembles that of Theorem~\ref{thm:Lp-integral}. Let $L_{p,\Sigma_0}(h,\natural,\omega^i)_n$ be the image modulo $\omega_{n,k-1}$ of $L_{p,\Sigma_0}(\ast,\natural,\omega^i,X)$, for $\natural \in\{\sharp, \flat\}$ and $h\in\{f,g\}$. Set
\[
\mathbf{L}_\natural :=  L_{p,\Sigma_0}(f,\natural,\omega^i,X) - L_{p,\Sigma_0}(g,\natural,\omega^i,X) .
\]
Choose $n$ sufficiently large so that $$\| \mathbf{L}_\natural \| = \| L_{p,\Sigma_0}(f,\natural,\omega^i)_n - L_{p,\Sigma_0}(g,\natural,\omega^i)_n \|$$ for both choices of $\natural$. Let us assume without loss of generality that $\| \mathbf{L}_\sharp \| \ge \| \mathbf{L}_\flat \|$ and suppose that $\| \mathbf{L}_\sharp \| > p^{-r}$.

For $n \gg 0$ even and $h \in \{ f, g \}$, we can write as in the proof of Theorem~\ref{thm:Lp-integral},
\[
C_{n,h}\begin{bmatrix}
    L_{p,\Sigma_0}(h,\sharp,\omega^i)_n\\ L_{p,\Sigma_0}(h,\flat,\omega^i)_n
\end{bmatrix}
=\begin{bmatrix}
   (U_h\Phi_{n,k-1}^- + p \Theta_h) L_{p,\Sigma_0}(h,\sharp,\omega^i)_n + p \Xi_h L_{p,\Sigma_0}(h,\flat,\omega^i)_n\\ \Psi_h
\end{bmatrix},
\]
where $U_h\in\cO\llbracket X\rrbracket^\times$, $\Theta_h,\Xi_h\in\cO[X]$.
By Lemma~\ref{lem:cong-Cn}, $\star_f\equiv \star_g\mod p^r$ for $\star\in\{U,\Theta,\Xi\}$. Thus, the first row of
\begin{equation}\label{diff-Cn}
C_{n,f}\begin{bmatrix}
    L_{p,\Sigma_0}(f,\sharp,\omega^i)_n\\ L_{p,\Sigma_0}(f,\flat,\omega^i)_n
\end{bmatrix}-
C_{n,g}\begin{bmatrix}
    L_{p,\Sigma_0}(g,\sharp,\omega^i)_n\\ L_{p,\Sigma_0}(g,\flat,\omega^i)_n
\end{bmatrix}.
\end{equation}
is equal to
\begin{gather*}
(U_g \Phi_{n,k-1}^- + p \Theta_g) \left( L_{p,\Sigma_0}(f, \sharp, \omega^i)_n - L_{p,\Sigma_0}(g, \sharp, \omega^i)_n \right) + p \Xi_g \left( L_{p,\Sigma_0}(f, \flat, \omega^i)_n - L_{p,\Sigma_0}(g, \flat, \omega^i)_n \right)\\
 + p^r \left( S L_{p,\Sigma_0}(f, \sharp, \omega^i)_n - T L_{p,\Sigma_0}(f, \flat, \omega^i)_n \right)
\end{gather*}
for some $S, T \in \mathcal{O}[X]$. We deduce that
\begin{itemize}
	\item{$\|p^r (S L_{p,\Sigma_0}(f, \sharp, \omega^i)_n - T L_{p,\Sigma_0}(f, \flat, \omega^i)_n)\|  \leq p^{-r} < \|\mathbf{L}_\sharp\|$;}
	\item{$\|p \Xi_g(L_{p,\Sigma_0}(f, \flat, \omega^i)_n - L_{p,\Sigma_0}(g, \flat, \omega^i)_n)\| < \| \mathbf{L}_\flat\| \leq \|\mathbf{L}_\sharp\|$, by assumption;}
	\item{$\|\mathbf{L}_\sharp\| = \|(U_g \Phi_{n,k-1}^- + p \Theta_g) (L_{p,\Sigma_0}(f, \sharp, \omega^i)_n - L_{p,\Sigma_0}(g, \sharp, \omega^i)_n)\|$, since $\Phi_{n,k-1}^-$ is monic.}
\end{itemize}
By the strong triangle inequality, $\| \text{first row of \eqref{diff-Cn}} \|=\|\mathbf{L}_\sharp\| > p^{-r}$. As a consequence of Theorem~\ref{thm:decomp}, we obtain $\|\mathbf{L}_\sharp\|=\|L_{p,\Sigma_0}(f, \omega^i)_n - L_{p,\Sigma_0}(g, \omega^i)_n\| > p^{-r}$. This contradicts Proposition~\ref{prop:cong-Pn}.  Therefore, we have $\|\mathbf{L}_\sharp\|\le p^{-r}$.

If instead we have $\| \mathbf{L}_\sharp \| \leq \| \mathbf{L}_\flat \|$, we assume $\|\mathbf{L}_\flat\|>p^r$; we can repeat the proof above but with $n \gg 0$ odd to deduce a contradiction.
\end{proof}


\section{Iwasawa invariants of the signed $p$-adic $L$-functions}\label{sec:4}
In this section, we prove Theorems~\ref{thmA}-\ref{thmC} presented in the introduction. Throughout, we assume that Assumption~\ref{asi:FL}-\ref{asi:cong} hold, which implies that Assumption~\ref{asi:main} holds for $\P=p$ and $r=1$.

We first recall the definition of Iwasawa invariants.
\begin{defn}
\textup{
For a non-zero element $\cF(X)= \sum_{m\ge0} c_m X^m\in\cO\llbracket X\rrbracket$, the $\mu$-invariant and $\lambda$-invariant of $\cF$ is defined as
\begin{align*}
	\mu(\cF) &:=  \min_m \ord_p(c_m); \ \text{and}\\
	\lambda(\cF) &: = \min\{ m: \ord_p(c_m) = \mu(\cF)\}.
\end{align*}
Given a finitely generated torsion $\cO\llbracket X\rrbracket$-module $M$, it is pseudo-isomorphic to a direct sum of cyclic modules $\bigoplus_i \cO\llbracket X\rrbracket/(\cF_i)$. Its characteristic ideal is given by
\[
\Char_{\cO\llbracket X\rrbracket}(M):=\left(\prod_i\cF_i\right)\subseteq\cO\llbracket X\rrbracket.
\]
We define the Iwasawa invariants of $M$ by $\mu(M):=\mu(\prod_i\cF_i)$ and $\lambda(M):=\lambda(\prod_i\cF_i)$.
}
    
\end{defn}

\subsection{Analytic side}\label{sec:analytic}

\begin{defn}\label{def:an-invs}
\textup{
For $0 \leq i \leq p-2$ and $\natural\in\{\sharp,\flat\}$ satisfying Assumption~\ref{asi:Lp-nonzero}, we define the analytic Iwasawa invariants
\begin{align*}
	\mu(f, \omega^i)_\text{\rm an}^{\natural} &:= \mu\left(L_p(f, \natural, \omega^i, X) \right),\\
	\lambda(f, \omega^i)_\text{\rm an}^{\natural} &: = \lambda(L_p(f, \natural, \omega^i, X) ).
\end{align*}
}
\end{defn}


We now prove Theorem~\ref{thmA}.

\begin{theorem}\label{thm:analytic}
Let $f$, $g$, $\natural \in \{ \sharp, \flat \}$ and $i\in\{0,\dots, p-2\}$ be chosen so that Assumptions~\ref{asi:FL}-\ref{asi:cong} hold. Let $\Sigma_0$ be the finite set of primes dividing $N_f N_g$. Then $\mu(f, \omega^i)_\text{\rm an}^\natural=0$ if and only if $\mu(g, \omega^i)_\text{\rm an}^\natural=0$, in which case
\[ 
\lambda(f, \omega^i)_\text{\rm an}^\natural = \lambda(g, \omega^i)_\text{\rm an}^\natural + \sum_{v \in \Sigma_0} \left(\mathbf{e}_v (g, \omega^i) - \mathbf{e}_v (f, \omega^i)\right), 
\]
where $\mathbf{e}_v(\ast,\omega^i)$ is the $\lambda$-invariant of $\cE_v(\ast,\omega^i)$.
\end{theorem}
\begin{proof}
It follows from \cite[Lemma~3.7.4]{EPW} that $\mathcal{E}_v(f, \omega^i)$ has trivial $\mu$-invariant. Thus, the first assertion of the theorem follows from Theorem~\ref{thm:cong-Lp} and \eqref{eq:Lp-imp}. To prove the second assertion, we note that since $L_{p,\Sigma_0}(f, \natural, \omega^i, X)$ is an element of $\cO\llbracket X\rrbracket$, we can apply the Weierstrass Preparation Theorem, which allows us to write it as $p^{\mu(f)_\mathrm{an}^\natural} \cdot D(X) \cdot U(X)$, where $U(X) \in \cO\llbracket X\rrbracket^\times$ while $D(X)$ is a distinguished polynomial in $\mathcal{O}[X]$. The lambda invariant $\lambda(f)_\text{\rm an}^\natural$ corresponds to the degree of $D(X)$, which is determined by the image of $L_p(f, \natural, \omega^i, X)$ modulo $p$. Computing the $\lambda$-invariant of $L_p(f,\natural,\omega^i,X)$ and $L_p(g,\natural,\omega^i,X)$ using the identity \eqref{eq:Lp-imp}, and then relating the two using Theorem~\ref{thm:cong-Lp}, we obtain
\[
\lambda(f, \omega^i)_\text{\rm an}^\natural + \sum_{v \in \Sigma_0} \lambda(\mathcal{E}(f, \omega^i)) = \lambda(g, \omega^i)_\text{\rm an}^\natural + \sum_{v \in \Sigma_0} \lambda(\mathcal{E}(g, \omega^i)),
\]
which concludes the proof.
\end{proof}

\subsection{Algebraic side} To prove Theorems~\ref{thmB} and \ref{thmC}, we need to put in place a few more definitions. As before, we introduce objects associated with the modular form $f$ with the understanding that the ones for $g$ are defined analogously. 

Let $V_f$ be the Deligne $p$-adic representation attached to a modular form $f$ constructed in \cite{deligne69} and let $T_f \subset V_f$ be a $G_{\Qp}$-stable $\cO$-lattice. We write $W_f = \Hom(T_f,E/\cO)(1)$, where the symbol $(m)$ represents the $m$-th Tate twist. 

Let $\Sigma$ be any finite set of primes containing $p$ and $\infty$, as well as the primes dividing $N_f$. Let $K$ be a number field. We write $K_\Sigma$ for the maximal extension of $K$ unramified outside of $\Sigma$. The \emph{$p$-primary Selmer group} of $f$ over a number field $K$ is given by the kernel of a local-global map
\begin{align*}
\Sel_p(K,W_f)
&= \ker\left( H^1(K, W_f) \to \prod_v \frac{H^1(K_v, W_f)}{H_f^1(K_v, W_f)} \right)\\
&\cong \ker\left( H^1(K_\Sigma/K, W_f) \to \prod_{\mathfrak{v} \mid v \in \Sigma} \frac{H^1(K_\mathfrak{v}, W_f)}{H_f^1(K_\mathfrak{v}, W_f)} \right),
\end{align*}
where
$H_f^1(K_\mathfrak{v}, W_f)$ is the Bloch--Kato subgroup defined in \cite[\S3]{blochkato}. If $\mathcal{K}$ is an infinite algebraic extension of $\QQ$, we define $\Sel_p(\mathcal{K},W_f)=\varinjlim_{K} \Sel_p(K,W_f)$, where $K$ runs over finite extensions of $\QQ$ contained inside $\mathcal{K}$ and the connecting maps are restrictions. 

While the $\cO\llbracket G_\infty\rrbracket$-module $\Sel_p(\QQ_\cyc,W_f)^\vee$ is finitely generated, it is not torsion when $f$ is non-ordinary at $p$. We recall the definition of sharp/flat Selmer groups, whose duals are expected to be torsion over $\cO\llbracket G_\infty\rrbracket$. They can be considered as the algebraic counterparts of the $p$-adic $L$-functions we studied previously. 

Let $\fp$ denote the unique prime of $\QQ_\infty$ lying above $p$. Then $\QQ_{\infty,\fp}$ can be identified with $\Qp(\mu_{p^\infty})$. We write
\[
\HIw(\Qp(\mu_{p^\infty}),T_f)=\varprojlim_n H^1(\Qp(\mu_n),T_f),
\]
where the connecting maps are corestrictions. As explained in \cite[\S3.2]{LLZ0} and \cite[\S2.3]{BFSuper}, the logarithmic matrix $M_{\log,f}$ decomposes the Perrin-Riou map on $\HIw(\Qp(\mu_{p^\infty},T_f)$ into a pair of Coleman maps
\[
\col^\natural:\HIw(\Qp(\mu_{p^\infty}),T_f)\longrightarrow \cO\llbracket \cG_\infty\rrbracket,
\]
where $\natural\in\{\sharp,\flat\}$. Furthermore, \cite[Theorem~2.14]{BFSuper} says that for $0\le i\le p-2$, there exists an explicit product of linear polynomials $\xi_{i,\natural}\in\cO[X]$ such that
\[
e_{\omega^i}\cdot\image(\col^\natural) \subset \xi_{i,\natural} \cdot \cO\llbracket G_\infty\rrbracket
\]
and the containment is of finite index. An explicit description of $\xi_{i,\natural}$ is given in \cite[\S5A]{LLZ0.5} (see also \cite[Appendix A]{harronlei}). Note that it only depends on $k$ and $p$.

\begin{defn}\label{def:selmergrps}
\textup{
The \emph{sharp and flat Selmer groups} of $f$ over $\QQ_\infty$ is given by
\[
\Sel_p^{\sharp/\flat}( \QQ_\infty,W_f) = \ker\left( \Sel_p( \QQ_\infty,W_f) \to  \frac{H^1(\QQ_{\infty,\mathfrak{p}},W_f)}{H_f^1(\QQ_{\infty,\mathfrak{p}}, W_f)^{\sharp/\flat}} \right),
\]
where $\fp$ denotes the unique prime of $\QQ_\infty$ above $p$ and  $H_f^1(\QQ_{\infty,\mathfrak{p}},W_f )^{\sharp/\flat}$ is defined to be the orthogonal complement of $\ker(\col^{\sharp/\flat})$ with respect to the local Tate pairing $\HIw(\QQ_{\infty,\fp},T_f) \times H^1(\QQ_{\infty,\fp},W_f) \rightarrow E/\cO$. We write $\mathcal{X}^{\sharp/\flat}(f)$ to denote the Pontryagin dual of $\Sel_p^{\sharp/\flat}(\QQ_\infty,W_f)$. For an integer $0\le i\le p-2$, we write $\mathcal{X}^{\sharp/\flat}(f)^{\omega^i}$ for the $\omega^i$-isotypic component of $\mathcal{X}^{\sharp/\flat}(f)$, which is a finitely generated $\cO\llbracket G_\infty\rrbracket$-module.
}
\end{defn}

\begin{remark}\label{rk:cotorsion}
\textup{
Recall from \cite[proof of Theorem 6.5]{LLZ0} that Assumption~\ref{asi:Lp-nonzero} implies that $\mathcal{X}^{\sharp/\flat}(f)^{\omega^i}$ is torsion over $\cO\llbracket G_\infty\rrbracket$. Note that Assumption~\ref{asi:Lp-nonzero} holds for any choice of $i$ and $\natural$ if $k>2$ or $a_p(f)=0$ (see \cite[Corollary~3.29 and Remark~3.30]{LLZ0}). 
}
\end{remark}

 Kato's main conjecture without $p$-adic zeta function, as formulated in \cite{kato04} is equivalent to:
\begin{conj}\label{IMC}
For $\natural\in\{\sharp,\flat\}$ and $0\le i \le p-2$, there is an equality of $\cO\llbracket G_\infty\rrbracket$-ideals
    \[
    \Char_{\cO\llbracket G_\infty\rrbracket}\left(\cX^\natural(f)^{\omega^i}\right)=\left(\frac{L_p(f,\natural,\omega^i,X)}{\xi_{i,\natural}}\right).
    \]
\end{conj}

\begin{defn}\label{def:alg-invs}
\textup{If Assumption~\ref{asi:Lp-nonzero} holds for $0 \leq i \leq p-2$ and $\natural\in\{\sharp,\flat\}$, we define the following algebraic Iwasawa invariants:
\begin{align*}
	\mu(f, \omega^i)_\text{\rm alg}^{\natural} &:= \mu\left( \mathcal{X}^\natural(f)^{\omega^i} \right) ,\\
	\lambda(f, \omega^i)_\text{\rm alg}^{\natural} &:= \lambda\left(  \mathcal{X}^\natural(f)^{\omega^i} \right) .
\end{align*}
}
\end{defn}

\begin{proposition}\label{prop:one-inclusion}
If Assumption~\ref{asi:Lp-nonzero} holds and $\mu(f,\omega^i)^\natural_\mathrm{alg}=\mu(f,\omega^i)^\natural_\mathrm{an}=0$, then
    \[
\frac{L_p(f,\natural,\omega^i,X)}{\xi_{i,\natural}}\in \Char_{\cO\llbracket G_\infty\rrbracket}\cX^\natural(f)^{\omega^i}.
    \]
\end{proposition}
\begin{proof}
It follows from \cite[proof of Corollary~6.8]{LLZ0} that there exists an integer $n_i^\natural$ such that
  \[
    p^{n_i^\natural}\cdot  \frac{L_p(f,\natural,\omega^i,X)}{\xi_{i,\natural}}\in \Char_{\cO\llbracket G_\infty\rrbracket}\cX^\natural(f)^{\omega^i}
    \]
since the $p$-adic $L$-function considered in \textit{loc. cit.} corresponds to the one studied in the present article by a constant given by the ratio of the respective periods utilized. The assumption on the vanishing of the $\mu$-invariants allows us to remove the constant $p^{n_i^\natural}$.
\end{proof}

Let $\Sigma_0$ be the finite set of primes dividing $N_fN_g$ as before. We  consider the $\Sigma_0$-imprimitive version of the sharp and flat Selmer groups obtained by ignoring the local conditions at primes in $\Sigma_0$
\[
\Sel_{p,\Sigma_0}^{\sharp/\flat} ( \QQ_\infty,W_f) = \ker\left( H^1(\QQ_{\Sigma_0}/\QQ_\infty, W_f) \to  \frac{H^1(\QQ_{\infty,\mathfrak{p}}, W_f)}{H_f^1(\QQ_{\infty,\mathfrak{p}}, W_f)^{\sharp/\flat}} \right)
\]
and its Pontryagin dual $\mathcal{X}^{\sharp/\flat}_{\Sigma_0}$. As discussed in \cite[\S5]{HL19}, we have the following isomorphism
\begin{equation}\label{eq:Selquotient}
\Sel_{p,\Sigma_0}^{\sharp/\flat} ( \QQ_\infty,W_f) / \Sel_p^{\sharp/\flat} ( \QQ_\infty,W_f) \cong \prod_{\mathfrak{v} | v \in \Sigma_0} \frac{H^1(\QQ_{\infty,\mathfrak{v}}, W_f)}{H_f^1(\QQ_{\infty,\mathfrak{v}}, W_f)}.
\end{equation}
In fact, $H_f^1(\QQ_{\infty,\mathfrak{v}}, W_f)=0$ whenever $\fv\nmid p$ since it is dual to 
\[
\varprojlim_n \frac{H^1(\QQ_v(\mu_{p^n}),T_f)}{H^1_f(\QQ_v(\mu_{p^n}),T_f)},
\]
which is zero by \cite[Lemma~6.2]{lei09}. Therefore, we have the equality
\begin{equation}
\Char_{\Lambda}\left( \mathcal{X}^{\sharp/\flat}(f)_{\Sigma_0}^{\omega^i} \right) = \Char_{\Lambda}\left( \mathcal{X}^{\sharp/\flat}(f)^{\omega^i} \right) \times \prod_{v \in \Sigma_0} \widehat{\mathcal{E}}_v(f,\omega^i),
\label{eq:selmer}    
\end{equation}
where $\widehat{\mathcal{E}}_v(f,\omega^i)$ is the characteristic ideal of the $\omega^i$-isotypic component of $\prod_{\mathfrak{v} \mid v} H^1(\QQ_{\infty,\fv},W_f)^\vee$.
\medskip

\begin{theorem}[Theorem~\ref{thmB}]\label{thm:algebraic}
Let $f$, $g$, $\omega^i$, $\natural\in\{\sharp,\flat\}$ and $\Sigma_0$ be as in Theorem~\ref{thm:analytic}. Then $\mathcal{X}^\natural(f)^{\omega^i}$ and $\mathcal{X}^\natural(f)^{\omega^i}$ are $\Lambda$-torsion. In addition, we have that $\mu(f, \omega^i)^\natural_\text{\rm alg}=0$ if and only if $\mu(g, \omega^i)^\natural_\text{\rm alg}=0$, in which case
\[ 
\lambda(f,\omega^i)_\text{\rm alg}^\natural = \lambda(g,\omega^i)_\text{\rm alg}^\natural + \sum_{v \in \Sigma}\left( \mathbf{e}_{v}(g,\omega^i) - \mathbf{e}_{v}(f,\omega^i)\right),
\]
where $\mathbf{e}_{v}(\ast,\omega^i)$ are defined as in the statement of Theorem~\ref{thm:analytic}.
\end{theorem}
\begin{proof}
Let $\ast\in\{f,g\}$.
The element $\mathcal{E}_v(\ast, \omega^i)$ defined in Section~\ref{sec:3} generates the characteristic ideal of $\widehat{\mathcal{E}}_v(\ast, \omega^i)$. Furthermore, the $\cO_f$-rank of $\prod_{\mathfrak{v} \mid v} H^1(\QQ_{\infty,\fv},W_f)^\vee$ coincides with the degree of $\mathcal{E}_v(\ast, \omega^i)$. These follow are derived from the proof of \cite[Proposition~2.4] {greenbergvatsal} carried out on p.37--38 of \textit{op. cit.}, after identifying the $\omega^i$-isotypic component of $H^1(\QQ_{\infty,\fv},W_\ast)^\vee$ with $H^1(\QQ_{\cyc,\fv},W_\ast(\omega^{-i}))^\vee$. It follows that the $\mu$-invariant of $\widehat{\mathcal{E}}_v(\ast, \omega^i)$ is zero, whereas its $\lambda$-invariant is equal to $\mathbf{e}_{v}(\ast,\omega^i)$.

Recall from Remark~\ref{rk:cotorsion} that Assumption~\ref{asi:Lp-nonzero} guarantees that $\mathcal{X}^{\natural}(f)^{\omega^i}$ and $\mathcal{X}^{\natural}(g)^{\omega^i}$ are torsion over $\cO\llbracket G_\infty\rrbracket$. The assertion on the $\mu$-invariants follows from \cite[Theorem~4.6]{HL19}. 

By \emph{loc.~cit.}, the $\lambda$-invariants of the dual $\Sigma_0$-imprimitive Selmer groups $\Sel_{\Sigma_0}^{\sharp/\flat}(\QQ_\infty,W_f)^\vee$ and $\Sel_{\Sigma_0}^{\sharp/\flat}( \QQ_\infty,W_g)^\vee$ coincide. Combined with \eqref{eq:selmer}, we obtain
\[
\lambda(f, \omega^i)_\text{\rm alg}^\natural + \sum_{v \in \Sigma_0} \lambda(\widehat{\mathcal{E}}_v(f, \omega^i)) = \lambda(g, \omega^i)_\text{\rm alg}^\natural + \sum_{v \in \Sigma_0} \lambda(\widehat{\mathcal{E}}_v(g, \omega^i)).
\]
We note that it was assumed in \textit{loc. cit.} that the eigenforms $f$ and $g$ have even weight $k \geq 2$. This condition is not used in the proof and carries over in verbatim in cases where $k$ is odd.
\end{proof}

We conclude the article with the following application to the Iwasawa main conjecture.
\begin{theorem}[Theorem~\ref{thmC}]\label{thm:IMC}
    Let $f$, $g$, $\natural$ and $i$ be chosen so that Assumptions~\ref{asi:FL}--\ref{asi:cong} hold. If $\mu(\ast, \omega^i)^\natural_\text{\rm alg}=\mu(\ast, \omega^i)^\natural_\text{\rm an}=0$ for $\ast\in\{f,g\}$, then Conjecture~\ref{IMC} holds for $\cX^\natural(f)^{\omega^i}$ if and only if it holds for $\cX^\natural(g)^{\omega^i}$.
\end{theorem}
\begin{proof}
    Proposition~\ref{prop:one-inclusion} and Assumption~\ref{asi:Lp-nonzero} imply that
    \[
    \lambda(\ast,\omega^i)^\natural_\mathrm{an}-\deg\xi_{i,\natural}\ge \lambda(\ast,\omega^i)^\natural_\mathrm{alg}.
    \]
    Furthermore, Conjecture~\ref{IMC} holds if and only if the equality holds. 

    After combining Theorems~\ref{thm:algebraic} and \ref{thm:analytic}, we have
    \[
    \lambda(f, \omega^i)_\text{\rm an}^\natural - \lambda(g, \omega^i)_\text{\rm an}^\natural =\lambda(f, \omega^i)_\text{\rm alg}^\natural -\lambda(g, \omega^i)_\text{\rm alg}^\natural .
    \]
    Hence, the result follows.
\end{proof}

\bibliographystyle{alpha}
\bibliography{references}

\end{document}